\documentclass[a4paper,12pt]{article}
\usepackage[utf8]{inputenc}
\usepackage{amsfonts}
\usepackage{amsthm}
\usepackage{amsmath}
\usepackage{amsfonts}
\usepackage{latexsym}
\usepackage{amssymb}
\usepackage{dsfont}
\usepackage{graphicx}
\usepackage[usenames]{color}

\newtheorem{fed}{Definition}[section]
\newtheorem*{fed*}{Definition}
\newtheorem*{feds*}{Definitions}
\newtheorem{teo}[fed]{Theorem}
\newtheorem*{teo*}{Theorem}
\newtheorem{lem}[fed]{Lemma}
\newtheorem{cor}[fed]{Corollary}
\newtheorem{pro}[fed]{Proposition}
\theoremstyle{definition}
\newtheorem{rem}[fed]{Remark}

\newtheorem*{rems*}{Remarks}

\newtheorem{exa}[fed]{Example}

\def\coma{\, , \, }

\def\py{\peso{and}}
\newcommand{\peso}[1]{ \quad \text{ #1 } \quad }
\def\Par{\big(\,}
\def\Pal{\big)\,}

\def\n0{n_{ \text{\rm \tiny o}}}

\def\M{\mathbb {M}}
\def\suml{\sum\limits}

\def\QEDP{\tag*{\QED}}

\def\bce{\begin{center}}
\def\ece{\end{center}}
\newcommand{\trivial}{\{0\}}

\def\py{\peso{and}}
\def\rk{\text{\rm rk}}
\def\noi{\noindent}
\def\cF{\mathcal F}
\def\cG{\mathcal G}
\def\QED{\hfill $\square$}
\def\EOE{\hfill $\triangle$}
\def\EOEP{\tag*{\EOE}}
\def\uno{\mathds{1}}
\def\bm{\left[\begin{array}}
\def\em{\end{array}\right]}
\def\ben{\begin{enumerate}}
\def\een{\end{enumerate}}
\def\bit{\begin{itemize}}
\def\eit{\end{itemize}}
\def\barr{\begin{array}}
\def\earr{\end{array}}
\def\igdef{\ \stackrel{\mbox{\tiny{def}}}{=}\ }

\def\precyl{\ \stackrel{\mbox{\tiny{Weyl}}}{\prec}\ }

\def\x{\mathbf{x}}
\def\y{\mathbf{y}}
\def\z{\mathbf{z}}

\def\la{\lambda}

\def\N{\mathbb{N}}
\def\R{\mathbb{R}}
\def\C{\mathbb{C}}
\def\I{\mathbb{I}}
\def\Z{\mathbb{Z}}

\def\cC{\mathcal{C}}

\def\cH{\mathcal{H}}
\def\cK{\mathcal{K}}

\def\cP{\mathcal{P}}

\def\cS{{\cal S}}

\def\cB{{\cal B}}

\def\cV{{\cal V}}
\def\cU{{\cal U}}

\def\vacio{\varnothing}
\def\orto{^\perp}
\def\inc{\subseteq}

\def\rai{^{1/2}}

\def\api{\langle}
\def\cpi{\rangle}

\def\ua{^\uparrow}
\def\da{^\downarrow}

\def\spr{\text{\rm Spr}}
\def\sprm{\text{\rm Spr}^+}

\DeclareMathOperator{\Preal}{Re} 
 \DeclareMathOperator{\tr}{tr}

\def\H{{\cal H}}

\def\beq{\begin{equation}}
\def\eeq{\end{equation}}

\def\pausa{\medskip\noi}

\def\Ax2{\,( S_{E(\cF)^\#_\cV})\hat{}_x }

\def\daua{\uparrow\downarrow}

\begin{document}

\title{Norm inequalities for the spectral spread\\ of Hermitian operators}
\author{Pedro Massey, Demetrio Stojanoff and Sebastian Zarate 
\footnote{Partially supported by CONICET
(PICT ANPCyT 1505/15) and  Universidad Nacional de La Plata (UNLP 11X829) 
 e-mail addresses:  massey@mate.unlp.edu.ar , demetrio@mate.unlp.edu.ar , seb4.zarate@gmail.com  }
\\
{\small Centro de Matem\'atica, FCE-UNLP,  La Plata
and IAM-CONICET, Argentina }}
\date{}

\maketitle

\begin{abstract}
In this work we introduce a new measure for the dispersion of the  spectral scale of a Hermitian (self-adjoint) operator acting on a separable infinite dimensional Hilbert space that we call spectral spread. Then, we obtain some submajorization inequalities involving the spectral spread of self-adjoint operators, that are related to Tao's inequalities for anti-diagonal blocks of positive operators, Kittaneh's commutator inequalities for positive operators and also related to the Arithmetic-Geometric mean inequality. In turn, these submajorization relations imply inequalities for unitarily invariant norms (in the compact case).
\end{abstract}

\noindent  MSC(2000) subject classification: 47A30, 47B10, 47B15.


\tableofcontents

\section{Introduction}

The development of inequalities involving spectral scales and generalized singular values of operators acting on a Hilbert space is a central topic in operator theory. The literature related to this research area is vast (see \cite{bha-Kitt90,bha-Kitt08,Hirz09,KitHir,Kitt,Kitt2008,Kitt09,Tao06,Zhan00,Zhan02} just to mention a few works strictly related to our present research). Therefore obtaining new inequalities, involving new concepts, is of interest in itself.

\pausa
In the matrix context (i.e. finite dimensional operator theory) A. Knyazev and A. Argentati introduced in \cite{AKFEM} an interesting measure for the dispersion of the eigenvalues of an Hermitian matrix called {\it spread}. They conjectured that the spread of an Hermitian matrix can be used to bound the so-called absolute variation of its Ritz values. In \cite{MSZ1} we obtained some inequalities that correspond to weak versions of Knyazev-Argentati's conjectures. At that point we realized that although natural and elegant, the spread of Hermitian matrices was not developed in the literature. Hence, in \cite{MSZ2} we made a systematic study of this notion. It turns out (see \cite{MSZ2}) that the spread of Hermitian matrices is related to several inequalities in terms of a pre-order relation known as submajorization (see \cite{bhatia}). For example, the spread allows us to obtain inequalities that are related to Tao's
inequalities \cite{Tao06} for the singular values of anti-diagonal blocks of positive matrices. 
 
 \pausa
On the other hand, in \cite{MSZ2} we showed that the spread of Hermitian matrices is also related to some commutator inequalities for generalized commutators (see also \cite{Hirz09,KitHir,Kitt,Kitt2008,Kitt09}). These results were applied in \cite{MSZ3} to obtain upper bounds for the absolute variations of Ritz values of Hermitian matrices that partially confirm Knyazev-Argentati's conjecture (although we point out that the original conjecture remains an open problem at this time).

\pausa
Motivated by Knyazev-Argentati's work \cite{AKFEM} and our previous works \cite{MSZ1,MSZ2,MSZ3} we introduce 
the { spectral spread} 
 of self-adjoint operators acting on a separable infinite dimensional Hilbert space $\cH$. 
In order to describe the spectral spread of a self-adjoint operator $A\in B(\cH)$, we consider its  spectral scale $\la(A)=(\la_i(A))_{i\in\Z_0}$ where $\Z_0=\Z\setminus\{0\}$ (see Definition \ref{defi espec scale}) defined by a ``min-max'' method. 
The numbers (entries) in this scale satisfy
$$
\la_{-i}(A)\leq \la_{-i-1}(A) \leq \la_{i+1}(A)\leq\la_{i}(A) \peso{for}i\in\N\,.
$$
For example, if $A$ is also compact, then the entries of the 
sequence $\la(A)$ are  the eigenvalues of $A$, 
in such a way that the numbers $\la_i(A)$ (for $i \in \N$) are the positive eigenvalues of 
$A$ counting multiplicities (or zero) arranged in  non-increasing order and 
$\la_{-i}(A)$ (for $i \in \N$) are the negative eigenvalues of $A$ counting multiplicities (or zero) arranged in 
 non-decreasing order.
Then, the  {\it spectral spread} of $A$, noted $\sprm(A)\in \ell^\infty(\N)$, is the non-negative and non-incresing sequence given by
$$
\sprm(A)=(\,\lambda_i(A)-\lambda_{-i}(A)\,)_{i\in \N} \,.
$$ 
After recalling some well known facts related to spectral scales and singular values (see Section \ref{Prelis}), 
we show some general properties of the spectral spread and its relation with the generalized singular values. 
Then we obtain several inequalities for the spectral spread, in terms of 
submajorization relations between generalized singular values of operators, and therefore for 
unitarily invariant norms associated to symmetrically normed operator ideals  in $B(\H)$
(which is a well known technique, see \cite{GoKre}). 
Our present results formally extend our previous results for the spread of Hermitian matrices related to Tao's and Kittaneh's inequalities 
 (see \cite{Tao06} and \cite{Kitt}). Moreover, we obtain some new (stronger) inequalities for commutators of self-adjoint operators in terms of submajorization (see Theorem \ref{teo doblesprm oplus}). 

\pausa
The fact that the spectral spread of self-adjoint operators is related to Tao's inequalities for singular values of anti-diagonal blocks suggests that the spectral spread can also be related to  Bhatia-Kittaneh's Arithmetic-Geometric mean inequalities \cite{bha-Kitt90,bha-Kitt08} (see \cite{Tao06}) and to Zhan's inequalities for the difference of positive operators \cite{Zhan00,Zhan02}. It turns out that this is the case; indeed, we develop
AGM-type inequalities,  and Zhan's type inequalities involving the spread of self-adjoint operators. On the other hand, we show that  many of the main inequalities for the spectral spread obtained in the present work are actually equivalent.

\pausa
The paper is organized as follows. In Section \ref{Prelis} we recall the basic facts about 
spectral theory for self-adjoint operators in $B(\cH)$ and
submajorization (between self-adjoint operators and between bounded real sequences). In Section \ref{sec spread} we introduce the spectral spread of a self-adjoint operator and obtain some basic results related with this notion. In Section \ref{sec 4} we obtain several submajorization inequalities for the spectral spread. Indeed, in Section \ref{A key inequality} we obtain an  inequality for the generalized singular values of anti-diagonal blocks of a self-adjoint operator that plays a key role throughout our work. In Section 
\ref{sac conmut ineq} we obtain inequalities for the generalized singular values of commutators of self-adjoint operators
in terms of the spectral spread. In Section \ref{AG} we develop several inequalities related to the Arithmetic-Geometric mean inequality in terms of submajorization relations and the spectral spread. In turn, these submajorization relations imply (in the compact case) inequalities with respect to unitarily invariant norms. In Section \ref{sec 53}  we show that 
 many of the main inequalities for the spread obtained in the present work are actually equivalent.

\section{Preliminaries}\label{Prelis}

In this section we introduce the basic notation and definitions used throughout our work.  

\pausa
{\bf Notation and terminology}. We let $\H$ be an infinite dimensional separable complex Hilbert space and we let $B(\H)$ be the algebra of bounded linear operators acting on  $\H$. In this case, $K(\H)\subset B(\H)$ denotes the ideal of compact operators and  $\cU(\H)\subset B(\H)$ denotes the group of unitary operators acting on $\H$. In what follows $K(\H)^{sa}$ and  $B(\H)^{sa}$  denote the real subspaces of self-adjoint and compact and self-adjoint operators, respectively. Also, we denote by $B(\H)^+$ the cone of positive operators and $K(\H)^+=B(\H)^+\cap K(\H)$. 

\pausa
We write $\N=\{1,2,\ldots\}$, $\I_k=\{1,...,k\}\subset \N$, for $k\in\N$ and $\I_\infty=\N$.
Also $\uno $ denotes the constant sequence 
with all its entries equal to one.

\subsection{Spectral scales and submajorization of self-adjoint operators}

In what follows, given $0\leq n$ we let $\cG(n,\cH)$ denote the  Grassmann manifold of $n$-dimensional subspaces of $\cH$ .  
Given a subspace $\cS \inc \cH$, we denote $\cS_1=\{x\in \cS: \|x\|=1\}$ the set of unit vectors in $\cS$. 
We also denote $\Z_0=\Z\setminus\{0\}$.

\begin{fed}\label{defi espec scale}\rm
Given $A\in B(\cH)^{sa}$ we define the  spectral scale of $A$ as the sequence $\la(A)=(\la_i(A))_{i\in\Z_0}$ determined by: for $i\in\N$ we let
$$
\la_i(A)=\inf_{\cS\in \cG(i-1,\cH)} \ \sup_{\psi\in \cS_1^\perp} \langle A\psi,\psi \rangle
\py
\la_{-i}(A)=\sup_{\cS\in \cG(i-1,\cH)} \ \inf_{\psi\in \cS_1 ^\perp} \langle A\psi,\psi \rangle\,.
$$
Notice that by construction, 
\beq\label{sangu}
\la_{-i}(A)\leq \la_{-(i+1)}(A) \leq \la_{i+1}(A)\leq\la_{i}(A) \ , \peso{for} i\in\N \ , 
\eeq
and $\la_i(-A)=-\la_{-i}(A)$ for $i\in\Z_0\,$. 
\EOE \end{fed}

\pausa 
The spectral scale of self-adjoint operators 
allows to develop the generalized singular values of arbitrary operators in $B(\cH)$. 

\begin{fed}\label{SV} \rm
Given $X\in B(\cH)$ we define the generalized singular values of $X$ as the sequence $s(X)=(s_i(X))_{i\in \N}$ determined by $s_i(X)=\la_i(|X|)$ for $i\in\N$, where $|X|=(X^*X)^{1/2}\in B(\cH)^+$.
 \EOE
\end{fed}

\pausa
We remark that the spectral scale of self-adjoint operators as well as the generalized singular values ($s$-numbers) of operators have been developed in the more general context of von Neumann algebras endowed with faithful semi-finite normal traces (see \cite{TFa,TFaKos,Kad04,Dpetz}).

\begin{rem}\label{rem en dim finit}
Let $\tilde {\cH}$ be a finite dimensional complex Hilbert space and let $\dim\tilde {\cH}=d$. 
Given $A\in B(\tilde {\cH})^{sa}$ we denote by $\mu(A) \in (\R^d)\da$ 
the vector of eigenvalues of $A$, counting multiplicities and arranged in 
 non-increasing order. 
In this case we can consider $\cG(n,\tilde {\cH})$, i.e. the Grassmann manifold of $n$-dimensional subspaces of $\tilde {\cH}$, 
for $0\leq n\leq d$.  Then, we  
can follow Definition \ref{defi espec scale} and 
set $(\la_i(A)\,)_{i\in\I_d}$ and similarly $(\la_{-i}(A)\,)_{i\in\I_d}$. 
Then 
\beq\label{dim fin}
(\la_i(A)\,)_{i\in\I_d}= \mu(A)  \py (\la_{-i}(A)\,)_{i\in\I_d}= \mu (A)\ua  
= \big( \mu_{d-i+1}(A)\,\big)_{i\in\I_d}\ .
\eeq
Further, in the Hermitian case the singular values of $A$ can be described 
as the non-increasing re-arrangement of the vector of eigenvalue modules i.e.,
$s(A)=(|\la_i(A)|)_{i\in\I_d}\da$. 

\pausa
We point out that this finite version of $\la(A)$ for $A\in B(\tilde \H)$ does not satisfy Eq. \eqref{sangu}. 
Nevertheless, we include the definition of $\la(A)$ since we need this notion in some cases (e.g. to describe 
the decreasing rearrangement of eigenvalues of 
$A_P = PA|_{R(P)}\in B(R(P))^{sa}$ where $P\in B(\H)$ is an orthogonal projection onto a finite dimensional subspace of $\H$, 
appearing in Theorem \ref{teo colecta}). On the other hand, 
the decreasing rearrangements of the eigenvalues of self-adjoint matrices (i.e. self-adjoint operators acting on finite dimensional Hilbert spaces) do share some fundamental properties with the spectral scale of self-adjoint operators acting on infinite dimensional Hilbert spaces. This last fact allows to obtain analogues of our results for self-adjoint matrices, with techniques similar to those included in the present work; we remark that some of these analogues for self-adjoint matrices are new. 
\EOE
\end{rem}

\begin{rem}\label{sing de transf}
Let $\cH_1$ and $\cH_2$ be two separable Hilbert 
spaces and assume that at least one of them has infinite dimension.
Consider $T\in B(\cH_1,\cH_2)$ a bounded linear transformation. In order to define the generalized singular values for $T$ 
we consider the following convention: 
\ben
\item[-] If $\cH_1$ is infinite dimensional, then we set $s_i(T)=\la_i(|T|)$ for $i\in \N$, where $|T|=(T^*T)^{1/2}\in B(\cH_1)$; 
\item[-] If $\dim \cH_1=k\geq 1$ then we set $s_i(T)=\la_i(|T|)$ for $i\in\I_k$ (here we use Remark \ref{rem en dim finit}) and $s_i(T)=0$, for $i\geq k+1$.
\een
Notice that in any case, $s(T)=(s_i(T))_{i\in\N}$ is a non-increasing sequence indexed by $\N$. Furthermore, with our present convention we always have that $s(T)=s(T^*)$.
\EOE
\end{rem}

\begin{rem}\label{caso K}
If  $A\in K(\H)^{sa}$ is compact (recall that $\dim \H = \infty$), it is easy to see that
 the entries of the sequence $\la(A)=(\la_i(A))_{i\in\Z_0}$ 
are also eigenvalues of $A$ (or zero), in such a way that the numbers $\la_i(A)$ 
(for $i \in \N$) are the positive eigenvalues of $A$ counting multiplicities  (or zero) arranged in 
 non-increasing order. Similarly the numbers 
$\la_{-i}(A)$ (for $i \in \N$) are the negative eigenvalues of $A$ counting multiplicities (or zero) arranged in 
the non-decreasing order. 

\pausa
Notice that if $T\in K(\cH )$, then $|T|\in  K(\H)^{sa}$. So that, the sequence 
$s(T)$ in Definition \ref{SV} is the usual sequence of singular values of $T$, considered as 
the non-increasing rearrangement of the eigenvalues of $|T|$, counting multiplicities. \EOE
\end{rem}

\begin{rem}\label{caso la}
For a general $A\in B(\H)^{sa}$, (see \cite[Section 3]{AMRS07}) it follows that  
$\lim_{i\rightarrow \infty}\la_{-i}(A)=\min \sigma_e(A)$ and 
$\lim_{i\rightarrow \infty}\la_i(A)=\max \sigma_e(A)$, where $\sigma_e(A) $ denotes the essential spectrum of $A$.
Hence, we have that 
\beq\label{cedea}
\barr{rl}
C(A) & \igdef \big\{\, \mu \in \R: \la_{-i}(A) \le \mu \le \la_{i}(A) \peso{for every} 
i \in \N\,\big\} \\&\\
& = \big[\, \min \sigma_e(A) \coma \max \sigma_e(A) \,\big] \not= \vacio \,. 
\earr
\eeq
Clearly the elements of 
$$
\Lambda^+(A) = \{\la \in \sigma (A) : \la >
\max \sigma_e(A)\}
$$ 
are eigenvalues with finite multiplicity (since $A-\la\, I $ is Fredholm). Also this set is countable and $\Lambda^+(A) = \{\la_i(A): i \in \N\}$  or  
 $|\Lambda^+(A) |=n $ and $\la_i(A) = \max \sigma_e(A)$ 
  for every $i>n$. A similar phenomenon happens with the sequence
$\la_{-i}(A)$ for $ i \in \N$,  and 
 $\min \sigma_e(A)$.
%
A consequence of this characterization of $\la(A)$ is 
that,  for every Hilbert space $\cK$ and every $B\in  B(\cK)^{sa} $, 
\beq\label{sumando B}
\sigma (B)\inc C(A) \implies
\la(A) = \la (A\oplus B) \ , \peso{where} A\oplus B \in  B(\H\oplus \cK)^{sa}  
\eeq
is the block diagonal operator determined by $A$ and $B$. Observe that Eq. \eqref{sumando B} 
holds for the operators $A\oplus \mu\, I_\cK$ for every  $\mu \in C(A)$. In particular, 
if we assume that $A \in  K(\H)^{sa} $ is compact, then automatically 
$C(A) = \trivial$ and $\la(A) = \la (A\oplus 0_\cK)$. 
\EOE
\end{rem}
\def\a{\mathbf a}

\begin{exa} Let $\a = (a_n)_{n\in \N} \in \ell^\infty_\R (\N)$ and $A = D_\a \in B(\ell^2 (\N)\,)^{sa}$
the diagonal multiplication operator as in \eqref{diago}. Then $C(A) = [\liminf \,\a\coma \limsup \, \a]$. This fact 
and Remark \ref{caso la} allows to compute the sequence $\la(A)$ in this case. For example, if 
$$a=(1+1/1\coma  -1+1/1\coma  1+1/2 \coma -1+1/2\coma  1+1/3\coma  -1+1/3\coma \ldots)\,,$$ we see that $\max\sigma_e(A)=1$, $\min\sigma_e(A)=-1$, $\Lambda^+(A)=\{1+1/n\ : \ n\geq 1\}$ and $\Lambda ^-(A)=\{\la\in\sigma \ : \ \la<\min\sigma_e(A)\}=\emptyset$. Hence, $\la_i(A)=1+1/i$ while $\la_{-i}(A)=-1$, for $i\in\N$.\EOE

\end{exa}

\pausa
In what follows, for $1\leq k\leq \infty$, we let
 $$\mathcal{P}_{k}(\H)=\{P\in B(\H)\, : \, P^2=P^*=P\,, \ \text{rk}(P)=k\}\,,$$ denote the subset of $B(\H)$ of all the orthogonal projections of rank $k$\,. 
Next, we collect several well known facts that we will need in the sequel (for details see \cite{AMRS07,GoKre}).

\begin{teo}\label{teo colecta}\rm
Let $A\in B(\cH)^{sa}$ and $P\in \cP_k(\cH)$ ($1\leq k\leq \infty$). 
\ben
\item Interlacing inequalities: 
Denote by $A_P \igdef PA|_{R(P)}\in B(R(P))^{sa}$. Then
\beq\label{PAP}
\la_j(A)\geq \la_j(A_P) \py \la_{-j}(A)\leq \la_{-j}(A_P)\peso{for} j\in\I_k\,.
\eeq
(In case $k\in\N$ we have to consider the Eq. \eqref{dim fin} in Remark \ref{rem en dim finit}).
\item For every $k\in\N$,
\beq
\sum_{i=1}^k \lambda_i(A)=\sup_{P\in \mathcal{P}_k} \tr(P\,A\,P) \py \sum_{i=1}^k\lambda_{-i}(A) =\inf_{P\in \mathcal{P}_k} \tr(P\, A\, P)\,.
\QEDP \eeq 
\een
\end{teo}

\pausa
We can now recall the notion of  submajorization and majorization between operators.

\begin{fed}\rm
Let $A,\,B\in B(\cH)^{sa}$. We say that $A$ is submajorized by $B$, denoted 
$$
A\prec_w B  \ , \peso{if}
\sum_{i=1}^k\la_i(A)\leq \sum_{i=1}^k\la_i(B)\peso{for every} k\in\N\,.$$
We say that $A$ is majorized by $B$, denoted $A\prec B$, if $A\prec_w B$ and $-A\prec_w -B$; equivalently, $A\prec B$ if and only if 
\beq
\sum_{i=1}^k\la_i(A)\leq \sum_{i=1}^k\la_i(B) \py  \sum_{i=1}^k\la_{-i}(A)\geq \sum_{i=1}^k\la_{-i}(B) \peso{for every} k\in\N\,. \EOEP
\eeq 
\end{fed}

\pausa
We can further consider the notion of (sub)majorization between sequences in $\ell^\infty(\M)$, where $\M=\N$ or $\M=\Z_0$. In order to do this, we consider the auxiliary Hilbert space $\ell^2(\M)$. Hence, given  
${\bf a}=(a_i)_{i\in\M}\in\ell_\R^{\infty}(\M)$ a bounded real sequence,  let 
$D_{\bf a}\in B(\ell^2(\M))^{sa}$ be determined by 
\beq\label{diago}
D_{\bf a}((\gamma_i)_{i\in\M})=(a_i\,\gamma_i)_{i\in\M} \peso{for} (\gamma_i)_{i\in\M}\in \ell^2(\M)\,.
\eeq

\begin{fed}\label{defi permutaciones sucs} \rm
Let ${\bf a}=(a_i)_{i\in\M_1}\in\ell_\R^{\infty}(\M_1)$ and ${\bf b}=(b_i)_{i\in\M_2}\in\ell_\R^{\infty}(\M_2)$ be real sequences, with $\M_i=\N$ or $\M_i=\Z_0$, for $i=1,2$. 
\ben
\item We let ${\bf a}\da=(a\da_i)_{i\geq 1}\in\ell^\infty_\R(\N)$ and ${\bf a}^{\daua}=(a^{\daua}_i)_{i\in\Z_0}\in \ell^\infty_\R(\Z_0)$ be given by 
$$
{a}_i\da=\la_i(D_{\bf a})\peso{for} \ i\geq 1 \py {a}_i^{\daua}=\la_i(D_{\bf a}) \peso{for} {i\in\Z_0}\,.
$$
\item We say that ${\bf a}$ is submajorized by ${\bf b}$, denoted ${\bf a}\prec_w{\bf b}$, if 
$$
\sum_{i=1}^k a_i\da \leq \sum_{i=1}^k b_i\da \peso{for} k\in\N\,.
$$
\item We say that ${\bf a}$ is majorized by ${\bf b}$, denoted ${\bf a}\prec{\bf b}$, if ${\bf a}\prec_w{\bf b}$ and $-{\bf a}\prec_w-{\bf b}$ i.e.,
\beq
\sum_{i=1}^k a_i^{\daua} \leq \sum_{i=1}^k b_i^{\daua} \py  \sum_{i=1}^k a_{-i}^{\daua} \geq \sum_{i=1}^k b_{-i}^{\daua}\peso{for} k\in\N\,.
\EOEP
\eeq 
\een
\end{fed}

\pausa 
As a consequence of Definition \ref{defi permutaciones sucs} and Theorem \ref{teo colecta}, given ${\bf a}=(a_n)_{n\in\M}\in \ell_{\R}^\infty(\M)$ then, 
$$
\barr{rl}
\suml_{i=1}^k a_i^{\daua} & =\sup\Big\{\,\suml_{i\in F}\,a_i :\ F\inc \M\coma |F|=k\,\Big\}\py 
\\
\suml_{i=1}^k a_{-i}^{\daua} &=\inf\Big\{\,\suml_{i\in F}\, a_i :
\ F\inc \M\coma |F|=k\,\Big\} \peso{for} k\in\N\ . 
\earr
$$
\pausa
Given sequences ${\bf a}=(a_n)_{n\in\N},\, {\bf b}=(b_n)_{n\in\N}\in \ell^\infty(\N)$ 
we let $({\bf a}, {\bf b})\in \ell^\infty(\Z_0)$ be the sequence determined by
\beq \label{de a dos}
({\bf a} \coma {\bf b})_n =\begin{cases} a_{-n} & \peso {if} n <0 \\ b_{n} &\peso {if} n>0 \end{cases}
\peso {for} n \in \Z_0 \,.  
\eeq
On the other hand, given ${\bf a}=(a_n)_{n\in\M},\,{\bf b}=(b_n)_{n\in\M}\in \ell^\infty(\M)$ 
we let 
\beq\label{prod}
{\bf a} \, \cdot\, {\bf b}\in \ell^\infty(\M) \peso{be given by} 
({\bf a} \, \cdot\, {\bf b})_n=a_n\,b_n  \ \ , \peso{for} n\in\M \ ,
\eeq where $\M=\N$ or $\M=\Z_0$.

\pausa
Submajorization relations appear in a natural way in operator theory. In the following result we collect some well known results related to this notion (see \cite{GoKre}).

\begin{teo}\label{teo rejunte1}\rm
Let $A,\,B , \,X,\,Y\in B(\cH)$. Then, the following relations hold:
\ben
\item Weyl's inequality for spectral scales: If $A\coma B\in B(\cH)^{sa}$ then $\lambda(A+B)\prec \lambda(A)+\lambda(B)\,.$ 
\item Weyl's inequality for generalized singular values: 
\beq\label{eq desi Weyl}
s(A+B)\prec_w s(A)+s(B) \ .
\eeq
\item $s_i(X\,A\,Y)\leq \|X\| \, \|Y\|\, s_i(A)$ for $i\geq 1$. In particular 
$$
s(X\,A\,Y)\prec_w \|X\| \, \|Y\|\, s(A) \py s(U\,A\,V) =s(A)
 \peso{for} U\coma V\in \cU(\cH)\ .
$$
\item\label{modulo} If $A\in B(\cH)^{sa}$ and  we let $|\la(A)|=(|\la_i(A)|)_{i\in\Z_0}$ then $s(A)=|\la(A)|\da$.
\qed\een 
\end{teo}

\begin{rem}\label{rem idealsimetrico}
Submajorization relations play a central role in the study of \textit{symmetrically normed operator ideals} (see \cite{GoKre} for a detailed exposition).  
A symmetrically normed operator ideal $\cC$ of $B(\H)$ is a proper two-sided ideal with a symmetric norm $N(\cdot)$, i.e, a norm with the following additional properties:

\begin{enumerate}
\item Given $A\in \cC$ and $D,E\in B(\H)$, then 
$N(D\,A\,E)\leq \|D\|\, N(A) \, \|E\|$; 
\item For any operator $A$ such that $\text{rk}(A)=1$, $N(A)=\|A\|=s_1(A)$. 
\end{enumerate}
 Moreover, any symmetric norm is unitarily invariant as a consequence of item 1; that is, if $U,V\in \cU(\H)$ then $N(UAV)=N(A)$, for every $A\in\cC$. Any such norm induces a gauge symmetric function $g_N$ 
defined on bounded sequences, such that $N(A) = g_N(s(A)\,)$. 

\pausa  In this context, we have that $\cC\subset K(\cH)$. Moreover, given $B\in \cC$ and $A\in K(\H)$ then 
\beq\label{la NUI} 
s(A)\prec_w s(B) \implies A\in \cC \py N(A)= g_N(s(A)\,)\leq g_N(s(B)\,) = N(B)\ . 
\eeq
For the sake of simplicity, in what follows we will call such a norm $N(\cdot)$  a \textit{unitarily invariant norm}. 
As examples of unitarily invariant norms we mention the Schatten $p$-norms, for $1\leq p<\infty$ associated to the Schatten ideals in $B(\H)$.
 \EOE

\end{rem}

\section{Spectral spread: definition and basic properties }\label{sec spread}

In this section we introduce and develop the first properties of the spectral spread for self-adjoint operators, which is motivated by the spread of self-adjoint matrices introduced by Knyazev and Argentati in \cite{AKFEM}.

\begin{fed}\label{defi spreads}\rm
Given $A\in B(\cH)^{sa}$ we define the  {\it full spectral spread} of $A$, denoted $\spr(A)\in \ell_\R^\infty(\Z_0)$ as the sequence
\beq\label{defi spread compactos}
\spr(A)\igdef(\,\lambda_i(A)-\lambda_{-i}(A)\,)_{i\in \Z_0}=\la(A)+\la(-A) \,.
\eeq  
We also consider the {\it spectral spread} of $A$, that is the non-negative and non-increasing sequence \beq\label{defi sprm}
\sprm(A)\igdef\big(\, \spr_i(A)\, \big)_{i\in\N} \in \ell^\infty(\N) \,.
\EOEP
\eeq
\end{fed}

\pausa 
Notice that, by Eq. \eqref{dim fin} and Eq. \eqref{defi spread compactos}, 
this definition of spectral spread essentially coincides with the matrix spread defined in \cite{AKFEM}
and \cite {MSZ2}. 
It is clear that the spectral spread of an operator in $B(\cH)^{sa}$ is a vector valued measure of the dispersion of its spectral scale.

\pausa
In the next result we collect some basic properties about the spectral spread in $B(\cH)^{sa}$. 

\begin{pro}\label{pro rejunte spread}
Let $A,B\in B(\cH)^{sa}$. The following properties holds:
\begin{enumerate}
\item $\spr(A)\in \ell_\R^\infty(\Z_0)$ is anti-symmetric $(\spr_{-j}(A)=-\spr_{j}(A))$; 
$ \sprm(A)\in \ell^\infty(\N) \cap \R_{\ge 0}^\N $.
\item The spectral spread is invariant under real translations i.e., 
for every $c\in \R$, 
$$
\spr(A+c\, I)=\spr(A) \py \sprm(A+c\,I)=\sprm(A)\ .
$$
\item For $c\in \R$ we have that  $\spr(c\, A)= |c|\, \spr(A)$. In particular, 
$\sprm(A)=\sprm(-A)$.
\item If $A\in K(\cH)^{sa}$ is compact, then $\sprm(A) = \sprm (A\oplus 0_\H )$.

\end{enumerate}
\end{pro}
\begin{proof} Items 1. and 2. are direct consequences of 
Eq. \eqref{defi spread compactos} in Definition \ref{defi spreads}.
Item 3. is a consequence of the following fact: given $A\in B(\cH)^{sa}$ then $\la(cA)=c\,\la(A)$ if $c\geq 0$ and $\la(cA)=c\,(\la_{-i}(A))_{i\in\Z_0}$ if $c<0$. Item 4. is a consequence 
of Remark \ref{caso la}.
\end{proof}

\pausa 
Observe that the equality $\sprm(A) = \sprm (A\oplus 0_\H )$ may be false for general self-adjoint 
operators (and also for matrices). For example $\sprm(I_\H) = 0$ but 
$\sprm(I_\H\oplus 0_\H) = \uno$, the sequence constantly equal to one. 
The following result describes several relations between the (full) spectral spread and singular values of self-adjoint operators.

\begin{pro}\label{pro rejunte spread2}
Let $A,\,B \in B(\cH)^{sa}$. 
\ben
\item  The following entry-wise inequalities hold:
\beq\label{spr vs s}
0\leq \sprm_i(A)\leq |\la_i(A)|+|\la _{-i}(A)| \le 2s_i(A) \peso{for every} i \in \N\, .
\eeq
In the positive case, we have that:
\beq\label{spr vs s pos}
A\in B(\H)^+\implies 
 \sprm_i(A) \leq  \la_i(A)= s_i(A)\peso{for every} i \in \N\, .
\eeq
\item Let $A\oplus A \in B(\H\oplus \H)^{sa}$ be given  by 
$A\oplus A= \bm{cc} A&0\\0&A\em$. Then 
\beq\label{AmasA}
\sprm(A\oplus A)=(\sprm(A),\sprm(A))\da\py 
\frac{1}{2}\, \sprm(A\oplus A)\prec_w s(A)\ .
\eeq
\item If $A\prec B$ then  $\sprm(A) \prec_w \sprm(B)\,.$
\item Additive Spread inequality: $\spr(A+B)\prec\spr(A)+\spr(B)$.
\een
\end{pro}

\proof
Since $s(A)=|\la(A)|\da$ (see Theorem \ref{teo rejunte1}), it follows 
that 
\beq\label{sisi}
\max\{|\la_i(A)|,|\la_{-i}(A)|\}\leq s_i(A) \peso{for every} i\in \N  \ .
\eeq
This proves the first part of item 1. The second part of item 1. follows from the fact that $\la_{-i}(A)\geq 0$ for $A\in B(\cH)^+$.
In order to show item 2. notice that 
the first claim in Eq. \eqref{AmasA} is straightforward. To show the second claim in Eq. \eqref{AmasA} 
we notice that
$$
\sum_{i=1}^{n}\sprm_i(A\oplus A)= \begin{cases} 2\, \suml_{i=1}^{k}\sprm_i(A) & \peso{if} n=2k \\
&\\
2\,  \suml_{i=1}^{k}\sprm_i(A)+ \sprm_{k+1} (A)   & \peso{if} n=2k +1 \end{cases}\ .
$$ 
Recall that $\suml_{i=1}^n s_i(A)=\sup \{\suml_{i\in F} |\lambda_i(A)|\ : \  F \subset \Z_0\, , \ |F|=n\}$. 
Then, using that 
$$
\sprm_{k+1} (A) = \lambda_{k+1}(A)- \lambda_{-(k+1)}(A)
\le 2\, \max\{|\la_{k+1}(A)|,|\la_{-(k+1)}(A)|\}
$$
and  that 
$$
2\, \sum_{i=1}^{k}\sprm_i(A)
\leq 2\,\sum_{i=1}^{k} |\lambda_i(A)|+|\lambda_{-i}(A)| \ ,
$$
we can easily prove the submajorization relation in Eq. \eqref{AmasA}. 

\pausa To show 3., fix $k\in \N$.  Since $\la(A)\prec\la(B)$, then
$$\sum_{i=1}^k \la_i(A) \leq \sum_{i=1}^k \la_i(B) \py
\sum_{i=1}^k \la_{-i}(A)\geq \sum_{i=1}^k \la_{-i}(B)
$$ 
$$
\implies \sum_{i=1}^k \sprm_i(A)=\sum_{i=1}^k \la_i(A)-\la_{-i}(A)
\leq \sum_{i=1}^k \la_i(B)-\la_{-i}(B)=\sum_{i=1}^k \sprm_i(B)\,.
$$

\pausa
To show item 4. 
notice that $\spr(A+B)=\la(A+B)+\la(-(A+B))$. Therefore
$$
\barr{rl}\spr(A+B) 
& \ \ = \ \lambda(A+B)+\lambda(-A-B) \\&
\\
& \precyl \lambda(A)+\la(B)+\lambda(-A)+\la(-B)
=\spr(A)+\spr(B)\ , 
\earr  
$$
where we have used Weyl's additive inequality for the spectral scale.
\qed

\section{Inequalities for the spectral spread}\label{sec 4}

In this section we obtain several submajorization inequalities for the spectral spread of self-adjoint operators. These inequalities show that the spectral spread is a natural measure of dispersion of the spectrum of self-adjoint operators.

\subsection{A key inequality}\label{A key inequality}

In  \cite{Tao06}  Tao showed that given a positive compact operator $F\in K(\H\oplus \cK)$ represented as a block matrix 
\beq\label{eq Taointro} 
F=\bm{cc}F_1&G\\
G^*&F_2 \em\peso{then} 2\,s_i(G)\leq s_i(F)\ , \peso{for} i\in\N\,.
\eeq 
It is natural to ask whether the inequalities in Eqs. 
\eqref{eq Taointro} hold in the more general case in which $F$ is a self-adjoint compact operator. It turns out that these 
inequalities fail in this more general setting (see \cite{MSZ2}).

\pausa 
The next submajorization inequality, which is our first main result of this section, is related to Tao's inequalities in Eq. \eqref{eq Taointro}; we point out that it will  play a key role in the rest of this work. We include the following proof in benefit of the reader; thanks to the definitions given in the previous sections, we can almost reproduce
the proof of the matrix case given in \cite{MSZ2} for this result: 

\begin{teo}\label{teo valecon2op}
Let $A\in B(\cH\oplus \cK)^{sa}$ be
such that $A=\bm{cc}
A_1& B\\
B^* & A_2
\em
 \barr{c}\cH\\\cK\earr  $ is the block representation for $A$. 
Then, $B\in B(\cK\coma\cH)$ (by construction) and
\beq \label{valecon2op}
2\,s(B)\prec_w \sprm(A)\,. 
\eeq
\end{teo}
\proof
Consider 
$U =\bm{cc}I&0 \\0&-I \em \barr {c}\H\\ \cK\earr \in \cU(\H\oplus \cK)$. 
Then $UA-AU\in B(\H\oplus \cK)$ and 
$$
s(UA-AU)=s(A-U^*AU)
=|\la(A-U^*AU)|\da\,,
$$ where we have used items 3. and 4. in Theorem \ref{teo rejunte1}.
By Weyl's inequality for spectral scales (see Theorem \ref{teo rejunte1}) we have that 
$$\la(A-U^*A\, U)\prec \la(A)+\la(- U^*A\, U)=\spr(A)\,,$$
since $\la(- U^*A\, U)=\la(-A)$. 
Using item 2. from Lemma \ref{lema desigrespetanmayo} we deduce that 
$$
s(UA-AU)=|\la(A-U^*A\,U)|	\da \prec_w
|\spr(A)|\da
=\Par \sprm(A)\coma \sprm(A)\,\Pal\da \,,
$$  where $\Par \sprm(A)\coma \sprm(A)\,\Pal\in \ell^\infty(\Z_0)$ is constructed as in 
Eq. \eqref{de a dos}. 
Straightforward computations and Proposition \ref{hat trick como en el futbol}  show that
$$
UA-AU=\bm{cc}0&2\,B \\-2\,B^*&0 \em \barr {c}\H\\ \cK\earr 
\implies s(UA-AU)=2\,\Par s(B)\coma s(B)\,\Pal\da\ ,
$$ 
and we conclude that 
\beq 
2\,\Par s(B)\coma s(B) \,\Pal\da
\prec_w \Par\sprm(A)\coma \sprm(A)\,\Pal\da\implies 2\,s(B)\prec_w\sprm(A)\ .
\QEDP
\eeq

\pausa
Although simple, the inequality in Eq. \eqref{valecon2op} is a useful result. Indeed, it plays a crucial role in the proof of Theorem \ref{teo doblesprm} below. On the other hand, this inequality can not be improved 
to an entry-wise inequality in the general case $A\in B(\cH)^{sa}$ (see \cite[Remark 2.9]{MSZ2}). 

\begin{cor}
With the notation of Theorem \ref{teo valecon2op}, assume further that $A\in K(\cH)^{sa}$.
Then, for any  unitarily invariant norm $N$ with gauge symmetric function $g_N$, we have that  
$$
2\,N(B)\leq g_N(\sprm(A))\,.
$$
\end{cor}

\proof
The inequality follows from Theorem \ref{teo valecon2op} and Remark \ref{rem idealsimetrico}.
\qed

\subsection{Commutator inequalities}\label{sac conmut ineq}

In \cite{Kitt} Kittaneh obtained the following singular value inequalities for commutators of positive compact operators: given $C,D\in K(\H)^+$ 
and a bounded operator $X\in B(\H) $, then 
\beq\label{kitt intro} 
s_i (CX-XD) \leq \| X \| \, s_i(C \oplus D) \ , \peso{ for } i\in \N\,,
\eeq
 where $\|\cdot\|$ denotes the operator (or spectral) norm. It turns out that Eq. \eqref{kitt intro} fails in case $C$ and $D$ are arbitrary self-adjoint compact operators (take  $C=X = I$ and $D=-C$). 

\pausa In what follows we obtain Theorem \ref{teo doblesprm oplus} related to Kittaneh's inequalities above that holds for self-adjoint operators.
We begin with the following observation. 
\begin{lem}\label{lematraza}
Let $A\coma B \in B(\H)^{sa}$ be such that $\rk(A)=n\in\N$. 
Then \rm
\beq \label{eq tr2}
 \tr(AB)\leq \tr\,\big(\,\la(A)\, \cdot\, \la(B)\, \big) \igdef 
\sum_{k\in \Z_0} \la_i(A) \, \la_i(B) \ .
\eeq
\end{lem}

\proof
Since $A$ has finite rank then it is a trace class operator. In particular, $A\in K(\cH)$ so Remark \ref{rem en dim finit} implies that the entries of $\la(A)$ are the eigenvalues of $A$ counting multiplicities and zeros; hence, 
$\tr (A) = \sum_{k\in \Z_0} \la_i(A) \,$. These last facts show that the inequality \eqref{eq tr2} is invariant if we replace $B$ by $B+ \mu I$, for any $\mu \in \R$; so we can assume that $B \in B(\cH)^+$.

\pausa
We also have that $\la_i(A)\geq 0$ and $\la_{-i}(A)\leq 0$, for $i\in\N$. Moreover, 
there is a ONB  for the range of $A$, denoted $R(A)$, say $\cB= \{x_i\}_{i=-m,\,i\neq 0}^r$ with $m+r=n$, $m,r\geq 0$,  such that $$A=\sum_{i=-m,\,i\neq 0}^r\lambda_i(A)\, x_i\otimes x_i\implies$$
\beq\label{con Bii}
\tr(A\, B)=\tr(B\,A)=\sum_{i=-m,\,i\neq 0}^r \lambda_i(A) \, \tr(B\,x_i\otimes x_i)
=\sum_{i=-m,\,i\neq 0}^r \lambda_i(A)\, \langle B\,x_i\coma x_i\rangle\,.
\eeq
Let $P$  denote the orthogonal projection onto $\text{Span} \{\cB\} 
= R(A)$. 
Denote by $B_P := PB|_{R(P)} \in \cB(R(A)\,)^+$ the compression of $B$ to $R(A)$. 	
Using the Schur-Horn theorem (see  \cite{bhatia} or \cite{AMRS07}) for $B_P$ and its matrix relative to $\cB$,  we get that 
its ``diagonal"
$$
d= \big(\, \langle Bx_{-m},x_{-m}\rangle,\ldots,\langle Bx_{-1},x_{-1}\rangle,\langle Bx_{1},x_{1}\rangle,\ldots,\langle Bx_{r},x_{r}\rangle \,\big )
\prec (\la_i(B_P)\,)_{i\in\I_n} \ .
$$  
Using Eq. \eqref{dim fin},  we can deduce that 
$d_+ :=(\langle Bx_{i},x_{i}\rangle)_{i\in\I_r}\prec_w   (\la_i(B_P))_{i\in\I_r}$ and, since 
also 
$$
-d \prec (-\la_i(B_P)\,)\da_{i\in\I_n}  \stackrel{\eqref{dim fin}}=  (-\la_{-i}(B_P)\,)_{i\in\I_n} \implies 
d_- := (-\langle Bx_{-i},x_{-i}\rangle)_{i\in\I_m}\prec_w (-\la_{-i}(B_P))_{i\in\I_m} \ .
$$
We now consider the auxiliary vectors $d_+\da=(a_i)_{i\in\I_r}$ and $d_-\da=(b_i)_{i\in\I_m}$ that are obtained 
from $d_+$ and $d_-$ defined above, by rearranging their entries in non-increasing order. 
Using that the vector $(\la_i(A))_{i\in\I_r}$ 
has non-negative entries and is arranged in non-increasing order then, by items 5. and 6. in Lemma \ref{lema desigrespetanmayo}, 
 we conclude that 
$$
(\la_i(A)\cdot \langle Bx_{i},x_{i}\rangle )_{i\in\I_r}\prec_w  (\la_i(A)\cdot a_i)_{i\in\I_r} \prec_w (\la_i(A)\cdot \la_i(B_P))_{i\in\I_r}$$ that implies that 
\beq\label{eq una para la traza1}
\sum_{i=1}^r \la_i(A)\cdot \langle Bx_{i},x_{i}\rangle\leq 
\sum_{i=1}^r \la_i(A)\cdot \la_i(B_P)\leq\sum_{i=1}^r \la_i(A)\cdot\la_i(B)\,,
\eeq where the last inequality follows from the interlacing inequalities \eqref{PAP}.
Similarly, using the interlacing inequalities, we see that 
$$
d_-\da =  (b_i)_{i\in\I_m} \prec_w(-\la_{-i}(B_P))_{i\in\I_m}
\stackrel{\eqref{PAP}}{\prec_w} (-\la_{-i}(B))_{i\in\I_m} \in (\R^m)\da \,.
$$
Since $(-\la_{-i}(A))_{i\in\I_m}\in (\R_{\ge 0}^m)\da $ 
then,
by item 2. in Lemma \ref{dos de Bh} 
we get that  
\beq\label{eq una para la traza3}
\sum_{i\in\I_m}-\la_{-i}(A)\cdot b_i\leq \sum_{i\in\I_m}\la_{-i}(A)\cdot\la_{-i}(B)\,.
\eeq
On the other hand, since 
the rearrangement of the vector $(\langle Bx_{-i},x_{-i}\rangle)_{i\in\I_m})$ in {\it non-decreasing} 
order coincides with $(-b_i)_{i\in\I_m}\in (\R_{\ge 0}^m)\ua$ then, by item 1. in Lemma \ref{dos de Bh} 
we conclude that 
$$
(\la_{-i}(A)\,b_i)_{i\in\I_m}\prec_w (-\la_{-i}(A)\,\langle Bx_{-i},x_{-i}\rangle)_{i\in\I_m}\,.
$$
The previous submajorization relation and Eq. \eqref{eq una para la traza3} imply that 
\beq\label{eq una para la traza2}
\sum_{i\in\I_m}\la_{-i}(A)\,\langle Bx_{-i},x_{-i}\rangle\leq \sum_{i\in\I_m}-\la_{-i}(A)\, b_i
\leq \sum_{i\in\I_m}\la_{-i}(A)\cdot\la_{-i}(B)\,.
\eeq 
Using the inequalities in Eqs. \eqref{eq una para la traza1} and 
\eqref{eq una para la traza2} we now see that  
$$
\tr(A\, B)\stackrel{\eqref{con Bii}}=\sum_{i=-m,\,i\neq 0}^r \lambda_i(A)\, \langle B\,x_i\coma x_i\rangle\leq 
\sum_{i=-m,\,i\neq 0}^r \lambda_i(A)\,\la_i(B)\, ,
$$
which completes the proof.
\qed

\begin{teo}\label{teo doblesprm}
Let $A,X\in B(\cH)^{sa}$. If we let ${\rm i}=\sqrt{-1}$ then 
\beq\label{eq agreg1}
\lambda \, \big(\, {\rm i}\, (A\,X-X\,A)\, \big) 
\prec_w \frac{1}{2}\,\sprm(A)\, \cdot\,  \sprm(X)\,.
\eeq
\end{teo}
\proof
Let $\varepsilon>0$ and $k\in\N$; by Theorem \ref{teo colecta}, there exists an orthogonal projection $P\in B(\H)$ with $k=\tr(P)$ such that 
$$
\sum_{j=1}^k \la_j\, \big(\, {\rm i}\, (A\,X-X\,A)\, \big)\leq \tr\, \big(\, {\rm i}\, (A\,X-X\,A)\,P\,\big)+\varepsilon\,.
$$
Moreover, since $XP-PX$ has finite rank then, by Lemma \ref{lematraza}, we get that
$$
\tr\, \big(\,{\rm i}\,(AX-XA) \,P\,\big)=
\tr\, \big(\,{\rm i}\,(XP-PX) \,A\,\big)\leq \tr\, \big(\,\la({\rm i}\,(XP-PX)\,)\cdot  \la(A)\,\big)\,.
$$
\pausa 
Now consider the block matrix representations induced by $P$:
$$
X=
\bm{cc}
 X_{11} & X_{12} \\ X_{12}^*& X_{22}
\em
\py P=\bm{cc}
 1 & 0 \\ 0& 0\em
\implies 
{\rm i}\,(XP-PX)={\rm i}\,
\bm{cc}
 0 & -X_{12} \\ X_{12}^*& 0\em\,.
$$
Denote $X_{12}=B\in K(\H)$; then by Proposition \ref{hat trick como en el futbol}
$\lambda ({\rm i}\,(XP-PX))=(s(B)\coma -s(B^*))^{\uparrow\downarrow}$. Now, Theorem \ref{teo valecon2op} implies that 
$s(B)\prec_w \frac{1}{2}\,\sprm(X)$. The previous fact together with by item 6. in 
Lemma \ref{lema desigrespetanmayo} show that
\beq \label{eq desi con submayo y productos} 
s(B) \cdot \sprm(A)\prec_w \frac{1}{2}\,\sprm(X)\cdot \sprm(A)\,.
\eeq
Moreover, if we let $k'\leq k$ be the 
number
of non-zero singular values of $B$ then 
\begin{eqnarray*}
\tr\, \big(\,\la({\rm i}\,(XP-PX)\,) \cdot \la(A)\,\big)&=&
 \sum_{j=1}^{k'} s_j(B)\,\la_j(A) - \sum_{j=1}^{k'} s_j(B) \, \la_{-j}(A)\\ \\ 
&=& \sum_{j=1}^{k'} s_j(B)\, \sprm_j(A) 
\stackrel{\eqref{eq desi con submayo y productos}}{\leq} 
\frac{1}{2} \sum_{j=1}^{k'} \sprm_j(X)\, \sprm_j(A)\,.
 \end{eqnarray*} 
Combining the previous arguments it is clear that
$$
\sum_{j=1}^k \la_j\, \big(\,{\rm i}\,(AX-XA)\,\big)\leq 
\frac{1}{2} \sum_{j=1}^{k'} \sprm_j(X)\, \sprm_j(A)+\varepsilon\leq 
\frac{1}{2} \sum_{j=1}^{k} \sprm_j(X)\, \sprm_j(A)+\varepsilon\,.
$$
Since $\varepsilon>0$ and $k\in\N$ were arbitrary, we get the submajorization relation in Eq. \eqref{eq agreg1}.
\qed

\pausa Although Theorem \ref{teo doblesprm} contains much information, its statement is rather technical. For example, the submajorization of Eq. \eqref{eq agreg1} only gives information about 
the ``positive part" of $\la({\rm i}\,(AX-XA)\,)$, namely   $ (\la_j({\rm i}\,(AX-XA)))_{j\in\N}\,$. 
The next result, which is a consequence of  Theorem \ref{teo doblesprm}, is more 
clear, and it has several direct implications (see Corollary \ref{cor commutator viejo} below).

\pausa
\begin{teo}\label{teo doblesprm oplus}
Let $A,X\in B(\cH)^{sa}$. Then 
\beq\label{eq teo prin conmut 1}
s(AX-XA)\prec_w\frac{1}{2}\, \sprm(A\oplus A)\, \cdot\,  \sprm(X\oplus X)\,.
\eeq
If we further assume that $A$ or $X\in K(\cH)^{sa}$ 
then, for any  unitarily invariant norm $N$ with gauge symmetric function $g_N$, 
\beq \label{eq teo prin conmut 2}
N(AX-XA)\leq \frac{1}{2}\, g_N\big(\,\sprm(A\oplus A)\cdot  \sprm(X\oplus X)\,\big)\,.
\eeq
\end{teo}
\proof
Using Eq. \eqref{eq agreg1} applied to
 $A, -X \in B(\cH)^{sa}$, 
$$
 \,\la(-{\rm i}\,(AX-XA))\prec_w \frac{1}{2}\,\sprm(A)\cdot \sprm(-X)=\frac{1}{2}\,\sprm(A)\cdot \sprm(X)\,,
$$ where ${\rm i}=\sqrt{-1}$ and we used that $\sprm(-X)=\sprm(X)$. 
By the comments after Definition \ref{defi espec scale}, item 4. in Theorem \ref{teo rejunte1}
 and item 4.  in Lemma \ref{lema desigrespetanmayo} we have that
\begin{eqnarray*}
s(AX-XA)&=&\big(\, (\la_j({\rm i}\,(\,AX-XA)\,)\,)_{j\in\N}\coma 
(\la_j(-\,{\rm i}\,(AX-XA)\,)_{j\in\N}\big)\da\\
\\ &\prec_w& \frac{1}{2}\,(\, \sprm(A)\cdot \sprm(X)\coma \sprm(A)\cdot \sprm(X)\, )\da\\ \\ &=&\frac{1}{2}\,
 \sprm(A\oplus A)\cdot  \sprm(X\oplus X)\,,
\end{eqnarray*} which proves Eq. \eqref{eq teo prin conmut 1}.
Finally, by Eq. \eqref{la NUI} in Remark \ref{rem idealsimetrico} we know that  
\eqref{eq teo prin conmut 1} $\implies$ \eqref{eq teo prin conmut 2}.
\qed

\pausa
A statement that is formally analogous to Theorem \ref{teo doblesprm oplus} is still valid in the matrix case, with the definition of spread given in  \cite{MSZ2}. Indeed, the proof of such claim can be obtained by a straightforward adaptation of the proof of Theorem \ref{teo doblesprm oplus}.
The next result, which is formally analogous to \cite[Theorem 3.1.]{MSZ2} (and played a central role in 
\cite{MSZ3}) is proved here with a new approach,
based on Theorem \ref{teo doblesprm oplus}.

\begin{cor}\label{cor commutator viejo}
Let $A,B\in B(\H)^{sa}$, $A \oplus B=\bm{cc}A&0\\
0&B \em\in B(\H\oplus \H)^{sa}$ and $X\in B(\H)$. Then, 
\beq\label{eq commutator viejo} 
s(AX-XB)\prec_w  \sprm(A\oplus B)\, \cdot\, s(X)\,.
\eeq
\end{cor}
\proof
First take $A=B$ and assume that $X\in B(\cH)^{sa}$. By Theorem \ref{teo doblesprm oplus}, 
Proposition \ref{pro rejunte spread2} (item 2.) and 
Lemma \ref{lema desigrespetanmayo} (item 6.), we have that
$$s(AX-XA)\prec_w \frac{1}{2} \, \sprm(A\oplus A) \cdot \sprm(X\oplus X)
\stackrel{\eqref{AmasA}}{\prec _w} \sprm(A\oplus A)\cdot  s(X)\,.$$
In the general case, let $C=A\oplus B\in B(\H\oplus \H)^{sa}$ and $\hat{X}=\bm{cc} 0&X\\
X^*&0 \em\in B(\H\oplus \H)^{sa}$\,.
Then $$C\hat{X}-\hat{X}C=\bm{cc} 0& AX-XB\\
(AX-XB)^* & 0 \em\,,$$ and, by the first part of the proof  $$s(C\hat{X}-\hat{X}C)=(s(AX-XB)\coma s(AX-XB)\,)\da\prec_w \sprm(C\oplus C)\cdot  s(\hat X)\,.$$ Notice that Eq. \eqref{eq commutator viejo} follows from the previous submajorization relation, since
\beq
s(\hat X)=(s(X)\coma s(X))\da \py 
\sprm(C\oplus C)=(\sprm(A\oplus B)\coma \sprm(A\oplus B))\da \ .
\QEDP
\eeq
\begin{rem}
Notice that inequality  \eqref{eq commutator viejo} cannot be improved to an entry-wise inequality, 
as Kittaneh's inequality in Eq. \eqref{kitt intro} for the positive case. 
For example, we can  consider 
$$A=\bm{cc}
	1&0\\
	0&1
\em
\peso{,}
B=\bm{cc}
	1&2\\
	2&1
\em
 \py X=\bm{cc}
2&1\\
1&2
\em 
\,,$$ all of them  embedded on $K(\H)^{sa}$ . Then 
$$
\la(A\oplus B)=\big(\, (\dots \coma 0\coma -1) \coma 
(3\coma 1\coma 1\coma 0\coma \dots )\, \big) \quad , \quad s(A\,X-X\,B)=(6\coma 2\coma 0\coma \dots)   
$$
and $s(X)=(3\coma 1\coma 0\coma \dots)$. 
Therefore   $s_2(A\,X-X\,B)=2>\sprm_2(A\oplus B)\,s_2(X)=1$. 
\EOE
\end{rem}

\pausa
The next result gives an upper bound for the generalized singular values of commutators of the form $AX-XB$, when $A,B,\,X\in B(\H)$ are arbitrary operators (for related results see \cite{KitHir}).

\begin{teo}\label{teo conmutadores largo}
Let $A,B $ and $X\in B(\H)$. Consider $A=A_1+{\rm i}\, A_2$, 
$B=B_1+{\rm i}\, B_2$ where $A_j,B_j\in B(\H)^{sa}$ for $j=1,2$ and ${\rm i}=\sqrt{-1}$. Then \beq\label{eq part real e im}
s(AX-XB)\prec_w \big(\sprm(A_1\oplus B_1)+\sprm(A_2\oplus B_2)\,\big)\cdot s(X)\,.
\eeq
\end{teo}

\proof
Notice that $$ AX-XB=(A_1+{\rm i}\, A_2)X-X(B_1+{\rm i}\, B_2)=A_1X-XB_1+{\rm i}\, (A_2X-XB_2)\,. $$ Then by Weyl's inequality for generalized singular values 
(see Eq. \eqref{eq desi Weyl}) and Corollary \ref{cor commutator viejo}, 
$$\barr{rcl}
s(AX-XB)&\prec_w& s(A_1X-XB_1)+s(A_2X-XB_2) \\  \\ 
&\prec_w& \sprm(A_1\oplus B_1) \cdot  s(X)+ \sprm(A_2\oplus B_2)\cdot  s(X)
\\ \\ 
&=&(\sprm(A_1\oplus B_1)+\sprm(A_2\oplus B_2)\,)\cdot  s(X)\,. 
\earr 
$$
\qed

\begin{cor}\label{coro conmutadores largo}
Let $A,B \in B(\H)$ and ${\rm i}= \sqrt{-1}$. Consider $A=A_1+{\rm i}\, A_2$, 
$B=B_1+{\rm i}\, B_2$ where $A_j,B_j\in B(\H)^{sa}$ for $j=1,2$. Let $a_j, a_j',b_j,b_j'\in \R$ be such that 
$$
a_j\, I \leq A_j\leq a_j' \, I \py b_j\, I \leq B_j\leq b_j'\, I \peso{for} j=1,2\,.$$ Then, for $X\in K(\H)$ and every unitarily invariant norm $N(\cdot)$ we have that 
$$ 
N(AX-XB)\leq \big(\,
\max\{a_1',b_1' \}-\min\{a_1,b_1\}+\max\{a_2',b_2'\}-\min\{a_2,b_2\} \,\big)\, N(X)\,.
$$
\end{cor}
\proof 
Notice that  since $a_j\, I \leq A_j\leq a_j'\, I $ and 
$b_j\, I \leq B_j\leq b_j'\, I $ for $j=1,2$, then 
$$ \sprm(A_j\oplus B_j)\leq \big(\,\max\{a_j',b_j' \}-\min\{a_j,b_j\}\,\big) \, \mathds{1}\, \text{ for } j=1,2\,.$$
From Theorem \ref{teo conmutadores largo} we now see that
\beq\label{eq part real e im coro}
s(AX-XB)\prec_w \,
\big(\,\max\{a_1',b_1' \}-\min\{a_1,b_1\}+\max\{a_2',b_2'\}-\min\{a_2,b_2\}\,\big)\, s(X)\,.
\eeq  Now the result follows from Eq. \eqref{eq part real e im coro} and Remark \ref{rem idealsimetrico}.
\qed

\begin{teo}\label{Teotomaestok} \rm
Let $A,\, X\in B(\H)^{sa}$ 
and  $U=e^{{\rm i}\,X}\in \cU(\H)$, where ${\rm i}=\sqrt{-1}$. Then \rm 
$$
s(A-U^*AU)\prec_w \frac{1}{2} \, \sprm(X\oplus X) \cdot \spr^+ (A\oplus A)\,.$$
\end{teo}
\proof  

\pausa Let $A(\cdot):[0,1]\rightarrow B(\cH)^{sa}$ be the smooth function given by 
$A(t)=e^{-{\rm i}\,tX}\,A\,e^{{\rm i}\,t X}$,  for $t\in [0,1]$. Notice that 
$A(0)=A$ and $A(1)=U^*A\,U$; using Weyl's inequality in Eq. \eqref{eq desi Weyl}, 
\beq \label{eq desi weyl en la curva conj unit}
\barr{rl}
s(A-U^*AU)\prec_w \suml_{j=0}^{m-1} s\big(\,A(\frac j m) - A(\frac {j+1} {m}) \,\big)
\peso{for every} m\in \N  \,.
\earr
\eeq 
Notice that $A(t+h)=e^{-{\rm i}\,t\,X}\,A(h)\,e^{{\rm i}\,t\,X}$ with $e^{{\rm i}\,t\,X}\in\cU(\H)$, for $t,\,h,\,t+h\in [0,1]$. Thus 
\beq\label{con m}
\barr {rl} s(A(\frac j m) - A(\frac {j+1} {m}))& 
=s(A - A(\frac 1 m) )\peso{for}j\in \I_{m-1}  \ . \earr
\eeq
Since $A-A(0)=0$ and $\frac{d}{dt} A(t)|_{t=0}={\rm i}\,(AX-XA)$ we get that 
\beq\label{AX-XA}
\barr{rl}
s(A - A(\frac 1 m) )&
=\frac 1 m \, s(AX-XA)+O(m)\peso{with}\lim\limits_{m\rightarrow \infty} m\,O(m)=0\,.
\earr
\eeq
Hence, by Theorem \ref{teo doblesprm oplus}  we have that 
\beq\label{eq el teo 41 se aplica aca}
\barr{rl}
s(A - A(\frac 1 m) )&\prec_w \frac{1}{2\,m} \, \sprm(X\oplus X)\, 
\cdot \, \sprm(A\oplus A)+ O(m)\,. 
\earr
\eeq 
Therefore, by Eqs. \eqref{eq desi weyl en la curva conj unit} and \eqref{con m} 
we have that, for sufficiently large $m$,
$$
s(A-U^*AU)\prec_w \frac{1}{2}\,\sprm(X\oplus X)\, \cdot \, \sprm(A\oplus A)+  m\,O(m)\,.
$$
The statement now follows by taking the limit $m\rightarrow \infty$ in the expression above.
\qed

\section{AGM-type inequalities and spectral spread}\label{AG}
Recall the 
Arithmetic-Geometric mean (AGM) inequality  for compact operators: given 
$A,B\in K(\H)$,  Bhatia and Kittaneh showed in \cite{bha-Kitt90} that:
\beq\label{Ec Bhakitta 2}
2\,s_i(A\,B^*)\leq s_i(A^*A+B^*B)\ , \peso{ for } i\in \N\, .
\eeq 
Then,  given 
$S,C\in B(\H)$ such that $C^*C+S^*S\leq I$ and $E\in K(\H)^+$, we get:
\beq\label{Ec consec Bhakitta 2}
2\,s_i(S\,E\,C^*)\leq s_i(E)\ , \peso{ for } i\in \N\,,
\eeq 
by taking $A=SE^{1/2}$ and $E^{1/2}C^*=B^*$ in Eq. \eqref{Ec Bhakitta 2}.  
Corach, Porta and Recht,  motivated by their study of the geometry in the context of operator algebras,  obtained in \cite{CPR} the following inequality with respect to a unitarily invariant norm $N(\cdot)$ 
\beq\label{CPR}
N(T)\leq \frac{1}{2} N(STS^{-1}+S^{-1}TS)\,,
\eeq where $T\in K(\H)$ is a compact operator and $S\in B(\H)$ is self-adjoint and invertible bounded operator. Later on, Bhatia and Davis showed in \cite{BhaDav} the following AGM-type inequality with respect to a unitarily invariant norm \beq\label{Ec BhaDav}
N(A^*XB)\leq \frac{1}{2} N(AA^*X+XBB^*)\,,
\eeq where $X\in K(\H)$  is a compact operator and $A,\,B\in B(\H)$ are bounded operators.
It turns out that inequalities in Eqs. \eqref{CPR} and \eqref{Ec BhaDav} are equivalent (by simple substitutions).

 \pausa
These AGM-type inequalities (both for singular values and for unitarily invariant norms) have been studied and extended in different contexts (\cite{Alda12,Audeh20a,AuKitt}); it turns out that they are related with deep geometric properties of operators \cite{ACS97,CPR,CPR2,CPR3}.

\pausa
Assume that $C,\,S\in B(\H)$ 
and let $E\in K(\H)^+$. As in Eq. \eqref{Ec consec Bhakitta 2}, we have that 
\beq \label{eq aud20}
s_j(SEC^*)=s_j(\,(SE^{1/2})(CE^{1/2})^*)\leq 
\frac{1}{2}s_j\,\big(\,E^{1/2}(S^*S+C^*C)E^{1/2}\,\big)
\, \peso{for} j\in \N \,.
\eeq 
Equations \eqref{Ec consec Bhakitta 2} and \eqref{eq aud20} were derived in \cite{Audeh20a}, where they were also shown to be stronger than 
the AGM-type inequalities obtained in \cite{Alda12} (see \cite{Hir05,Hir16} for related singular values inequalities). 
In particular, 
if $C^*C+S^*S \leq I$, for any unitarily invariant norm $N$ we have that 
\beq \label{eq aud202}
N(SEC^*)\leq \frac{1}{2}N(E)\,.
\eeq From Eqs. \eqref{eq aud20} and \eqref{eq aud202} it is possible to derive more general AGM-type inequalities. Nevertheless, Eq. 
\eqref{eq aud202} fails for arbitrary self-adjoint $E\in K(\H)^{sa}$: 

\begin{exa} As in previous examples, we shall use that a matrix $A\in \mathcal{M}_n(\C)$ can be embedded as a finite rank operator, which allow us to 
build counterexamples using matrices. 
Now we see that Eq. \eqref{eq aud202} is false if $E \not \ge 0$. 
Consider $N(X)=\|X\|_2=(\tr(X^*X))^{1/2}$ the Frobenius norm (which is unitarily invariant). 
Let 
$$S=\bm{cc}\sin(\pi/3)&0\\
		0&\sin(\pi/5)
	\em\,, \, \, E= \bm{cc}
0&1\\
1&0
\em \py C=\bm{cc}
\cos(\pi/3)&0\\
0&\cos(\pi/5)
\em\,.$$

\pausa
Then $C^*C+S^*S=I$ and $\|S\,E\,C^*\|_2\approx 0.7598 > \frac{\sqrt 2}2 
= \frac{1}{2}\,\|E\|_2\,$. 
\EOE
\end{exa}

\subsection{AGM-type inequalities: the general self-adjoint case}
\pausa
In what follows we obtain a generalization of Eq. \eqref{eq aud202} for arbitrary self-adjoint $E\in B(\H)^{sa}$; these are new inequalities that involve upper bounds in terms of the spectral spread of self-adjoint operators in $B(\H)$. We point out that our results are based on the (weaker) submajorization relations; 
 nevertheless, these results imply inequalities with respect to unitarily invariant norms as in Eq. \eqref{eq aud202}.

\begin{pro}\label{pro para gen AM}
Let $C,\,S\in B(\H)$  be such that $C^*C+S^*S =P =P^2$ 
and let $E\in B(\H)^{sa}$.
Then
\beq\label{eq con spread1 tuti} 
2\,s(SEC^*)\prec_w \sprm(PEP \oplus 0_\H )\ .
\eeq
Note that if $\dim \ker P = \infty$, we can rewrite $2\,s(SEC^*)\prec_w \sprm(PEP)$. 
\end{pro}
\proof
By considering the polar decomposition of $S$ and $C$ and 
Theorem \ref{teo rejunte1}, 
we can assume that $S,\,C\in B(\H)^+$. 
Let $\cK = R(P)$ and consider the Hilbert space $\cH\oplus \cK$.
We consider an orthogonal decomposition $\cH\oplus \cK=\cK\oplus \cK^\perp\oplus \cK$; this
decomposition allows us to represent operators $T\in B(\cH\oplus \cK)$ as $3\times 3$ 
block matrices, in the usual sense. 
Since the compressions $C_P^2+S_P^2=I_P\in B(\cK)$ then the operator $U\in B(\cH\oplus \cK)$ whose block 
representation is given by 
$$
\bm{ccc}
C_P & 0 & -S_{P} \\
0 & I  & 0 \\
S_P & 0 & C_P 
\em
$$ is unitary. Furthermore, it is straightforward to check that the block 
representations  of $PEP\oplus 0_\cK,\,U (PEP\oplus 0_\cK)U^*\in B(\cH\oplus \cK)^{sa}$ are given by
$$
PEP\oplus 0_\cK= \bm{ccc}
E_P & 0 & 0 \\
0 & 0 & 0 \\
0 & 0 & 0 
\em
\py 
U (PEP\oplus 0_\cK)U^*=
\bm{cc|c}
C_PE_PC_P & 0 & C_PE_PS_P \\
0 & 0 & 0 \\
\hline
S_PE_PC_P & 0 & S_PE_PS_P 
\em
\,.
$$ 
where $E_P\in B(\cK)$ is the compression of $E$ to the subspace $\cK$. Since $C=CP = PCP$ and $S=SP= PSP$, we get that $S_PE_PC_P=(SEC)_P$ and hence, the anti-diagonal block in the block representation above 
$$[S_PE_PC_P\ \ \ 0]= [(SEC)_P \ \ \ 0]=  SEC\in B(\cH,\cK)\,.$$
Notice that in the equality above, we need to restrict the co-domain of $SEC\in B(\cH)$; yet, this restriction
does not affect the generalized singular values (see Remark \ref{sing de transf}). Hence, as a consequence of Theorem \ref{teo valecon2op} we get that 
\beq \nonumber
2\,s(SEC)\prec_w \sprm(U (PEP\oplus 0_\cK)U^*)=\sprm(PEP\oplus 0_\cK)\,.
\eeq 
Finally, notice that we always get that $\sprm(PEP\oplus 0_\cK)=\sprm(PEP\oplus 0_\cH)$. Indeed, if $\dim\cK=\dim\cH$ this is clear; in case $\dim\cK<\infty$ then $PEP\in K(\cH)^{sa}$ and therefore, $\la(PEP)=\la(PEP\oplus 0_\cK)=\la(PEP\oplus 0_\cH)$ which proves the identity between spectral spreads above. \qed

\begin{rem}\label{sin sumar 0} \rm 

\pausa
The formulation of Proposition \ref{pro para gen AM}
is sharp in this general case
(see Remark \ref{rem mas cero}  below). 
Nevertheless, if $E\in K(\H)^{sa}$ 
we get a better 
estimate 
(see Corollary \ref{pro para gen AM comp}). 
\EOE
\end{rem}

\pausa
Proposition \ref{teo relaciones AG con spred4} and Corollary \ref{coro AG} below complement Proposition \ref{pro para gen AM}. 

\begin{pro}\label{teo relaciones AG con spred4}
Let $C,\,S\in B(\H)^+$ be such that $C^2+S^2=P = P^2$. 
If $E_1,\,E_2\in B(\H)^{sa}$ then
\beq\label{eq nueva para comp1}
s(SE_1C+CE_2S)\prec_w \frac 1 2 \ \sprm(PE_1P\oplus -PE_2P)\,.
\eeq
\end{pro}
\proof
We consider the auxiliary Hilbert space
$\cH\oplus \cH$ together with its orthogonal decomposition  
$$\cH\oplus \cH=\cK\oplus \cK^\perp \oplus \cK\oplus \cK^\perp\,,$$
where $\cK=R(P)$. We also consider $U\in \cU(\cH\oplus \cH)$, whose
block representation is given by 
$$
\bm{cccc}
C_P & 0 & -S_{P}&0 \\
0 & I  & 0&0 \\
S_P & 0 & C_P&0\\
0& 0& 0& I
\em
$$
 It is straightforward to check that 
the block representations of $PE_1P\oplus -PE_2P\in B(\cH\oplus \cH)^{sa}$ and $U(PE_1P\oplus -PE_2 P)U^*$ are given by
$$
\bm{cccc}
(E_1)_P & 0 & 0& 0 \\
0 & 0 & 0 & 0\\
0 & 0 & -(E_2)_P & 0\\
0 & 0 & 0& 0 \\
\em
\py 
\bm{cc|cc}
(C E_1 C-SE_2S)_P & 0 & (CE_1S+SE_2C)_P & 0\\
0 & 0 & 0 & 0\\
\hline
(SE_1C+CE_2S)_P & 0 & (SE_1S-CE_2C)_P & 0\\
0 & 0 & 0 & 0
\em
\,.
$$ 
Arguing as in the proof of Proposition \ref{pro para gen AM}
we see that
$$
s(\bm{cc} (SE_1C+CE_2S)_P & 0 \\
0 & 0 \em)=s(SE_1C+CE_2S)\,.
$$
Hence, as a consequence of Theorem \ref{teo valecon2op} we get that 
\beq \nonumber
2\,s(SE_1C+CE_2S)\prec_w \sprm(U(PE_1P\oplus -PE_2P)U^*)=\sprm(PE_1P\oplus- PE_2P)\,.\QEDP\eeq

\pausa It is easy to see that, if $E\in B(\H)^{sa}$ then 
\beq\label{eq nueva para comp2}
\sprm(E\oplus - E) = \big(\, 2\, |\la (E)|\,\big) \da = 2 \ s(E) \ . 
\eeq
We use this in the following inequality, valid for a general $E\in B(\H)^{sa}$:

\begin{cor}\label{coro AG}
Let 
$C,\,S\in  B(\H)^+$ be such that $C^2+S^2=P = P^2$. 
If $E\in B(\H)^{sa}$ then
$$
s(\,\Preal(SEC)\,)\prec_w 
\frac 1 2 \ s(E)\,.
$$
\end{cor}
\proof
Notice that $\Preal(SEC)=\frac{SEC+CES}{2}$. 
By Proposition \ref{teo relaciones AG con spred4} 
with $E_1=E_2=E$, we get 
\beq
s(\,\Preal(SEC)\,)\prec_w \frac 1 4 \ \sprm(PEP\oplus - PEP) \stackrel{\eqref{eq nueva para comp2}}= 
\frac 1 2 \ s(PEP) \prec_w \frac12\ s(E) \ . \QEDP
\eeq

\subsection{AGM-type inequalities: the compact self-adjoint case}

\pausa
We begin by 
reformulating  some facts about spectral scales, submajorization and spectral spread for compact self-adjoint operators. 
 Recall the definition of spectral scale of self-adjoint operators given in Definition \ref{defi espec scale}. 
Also recall from Remark \ref {caso K} that 
if $A\in K(\H)^{sa}$ is compact,   then the entries of the 
sequence $\la(A)=(\la_i(A))_{i\in\Z_0}$ 
are also eigenvalues of $A$ (or zero), in such a way that the numbers $\la_i(A)$, for $i \in \N$, are the positive eigenvalues of $A$ counting multiplicities (or zero) arranged in 
non-increasing order. Similarly the numbers 
$\la_{-i}(A)$, for $i \in \N$, are the negative eigenvalues of $A$ counting multiplicities (or zero) arranged in 
non-decreasing order.

\pausa
Now we show some properties of the spectral spread in the compact case: 

\begin{pro}\label{sprm comp}
Let $A\in K(\cH)^{sa}$. Then 
\ben
\item 
If $P\in P_k(\cH)$ ($1\leq k\leq \infty$) then 
\beq\label{PAPK}
\la_j(A)\geq \la_j(PAP) \py \la_{-j}(A)\leq \la_{-j}(PAP)\peso{for} j\in\I_k\,.
\eeq
(Notice that this is a reformulation of Eq. \eqref{PAP} for the compact case).
\item 
For every $i \in \N $,
\beq\label{spr con P}
\sprm_i (PAP)\le \sprm_i(A) \implies 
\sprm(PAP)\prec_w \sprm(A) \ .
\eeq
\een 
\end{pro}
\proof
Consider $A_P \igdef PA|_{R(P)} \in K(R(P)\,)^{sa}$. Then $PAP = A_P \oplus 0_{\ker P}\,$.
In case that  $k = \dim R(P) = \infty$, we can apply Remark \ref{caso la}, so that for every $j \in \N\,$ 
$$
\la_j(PAP) = \la_j(A_P) 
\ge 0 \py \la_{-j}(PAP)=  \la_{-j}(A_P)  \le 0 \ .
$$
If $k<\infty$ then by Remark \ref{rem en dim finit}, 
$\la_j(PAP) = \max\{\la_j(A_P)\coma 0\}$ and 
$\la_{-j}(PAP)=  \min\{\la_{-j}(A_P)\coma 0\}$ for $j\in\I_k\,$. 
Hence Eq. \eqref{PAPK} follows in both cases from
the interlacing inequalities of Eq. \eqref{PAP}. 

\pausa
In particular, we now see that $\sprm_j(PAP)\leq \sprm_j(A)$, for $j\in\I_k\,$. 
If $k<\infty$, for every $j>k$ we have that $\la_j(PAP)=\la_{-j}(PAP)=0$ and then  $\sprm_j(PAP)=0\leq \sprm_j(A)$. The submajorization relation $\sprm(PAP)\prec_w \sprm(A)$ now follows directly from these facts. 
\QED

\pausa
Now we can reformulate Proposition \ref{pro para gen AM} for the compact case: 

\begin{cor}\label{pro para gen AM comp}
Let $C,\,S\in B(\H)$  be such that $C^*C+S^*S =P$, where $P=P^2$. 
If $E\in K(\H)^{sa}$ then
\beq\label{eq con spread1} 
2\,s(SEC^*)\prec_w \sprm(E)\,.
\eeq
\end{cor}
\proof 
By Proposition \ref{pro para gen AM} and  
item 4. of Proposition \ref{pro rejunte spread}, since $PEP\in K(\H)^{sa}$, 
$$
2\,s(SEC^*)\prec_w \sprm(PEP\oplus 0_\H) = \sprm(PEP)\ .
$$
Also, by Eq. \eqref{spr con P},
$\sprm(PEP)\prec_w \sprm(E)$. This shows Eq. \eqref{eq con spread1}. 
\QED

\begin{rem}\label{rem mas cero}
The statement of Corollary \ref{pro para gen AM comp} can fail
in the non compact case, 
 since $\spr^+(PEP)\prec_w \spr^+(E)$ does not hold in general  (take $P\neq I$ and $E=I$).
For example, if  $E =  I$ we have that $\sprm(E) = 0$. But taking $C = S = \frac1{\sqrt{2}} \, I$  we get  $s(SEC^*) = \frac 12 \uno$. 
So that Eq. \eqref{eq con spread1} fails in this case.
 Notice that even in this extreme case, 
Eq. \eqref{eq con spread1 tuti} in Proposition \ref{pro para gen AM} is still true, 
because $\sprm (I\oplus 0_\H) = \uno$.

\pausa	
On the other hand, with the notation of Corollary \ref{pro para gen AM comp}, if $E\in K(\H)^+$ then
$\sprm(E) \stackrel{\eqref{spr vs s pos}}=s(E)$.
Hence, Corollary \ref{pro para gen AM comp} implies that $N(SEC)\leq \frac 1 2 \,N(E)$ and we recover Eq. \eqref{eq aud202}. 

\pausa 
In case $E\in K(\H)^{sa}\setminus K(\H)^+$ then it turns out that 
$s(E)\prec\sprm(E)$, with strict majorization i.e., if $N$ is a strictly convex unitarily invariant norm then $$N(E)<g_N(\sprm(E))\,.$$ 
For example, if we consider $N(X)=\|X\|_2=(\tr(X^*X))^{1/2}$ then we have that 
$$\|E\|_2^2=\sum_{j\in\Z_0}\la_j(E)^2 < \sum_{j\in\N}(\la_j(E)-\la_{-j}(E))^2 
=	g_{||\cdot||_2}(\sprm(E))^2\,,$$
for every Hilbert-Schmidt self-adjoint operator $E\notin K(\cH)^+$.
\EOE
\end{rem}

\begin{cor}
Let $C,\,S\in B(\H)^+$ be such that $C^2+S^2=P = P^2$. 
If $E_1,\,E_2\in K(\H)^{sa}$,  then
\beq\label{eq nueva para comp3}
s(SE_1C+CE_2S)\prec_w \frac 1 2 \ \sprm(E_1\oplus -E_2)\,.
\eeq
\end{cor}
\proof
This is a straightforward consequence of Corollary \ref{teo relaciones AG con spred4}, the fact that 
$$PE_1P\oplus PE_2P=(P\oplus P) (E_1\oplus E_2)(P\oplus P)\in K(\cH\oplus\cH)^{sa}$$ and item 2. in Proposition \ref{sprm comp}.
\qed

\pausa
Now we can state our main result on AGM-type inequalities for unitarily invariant norms.

\begin{teo}\label{teo generalized AG ineq}
Let $A,\,B\in B(\H)$ and let $E\in K(\H)^{sa}$. 
Then 
\beq\label{eq con spread2} 
s(A\,E\,B^*)\prec_w \frac 1 2 \ \sprm \, \big(\,(A^*A+B^*B)^{1/2}\,E\,(A^*A+B^*B)^{1/2}\,\big)\,.
\eeq 
\end{teo}

\begin{proof} 
By Douglas' theorem (cited as Theorem \ref{dou} in the Appendix),
the operator inequality $A^*A \le A^*A+B^*B$ shows that 
the (linear) operator equations
\beq \label{eq op eq1}
A=S\,(A^*A+B^*B)^{1/2} \py B=C\,(A^*A+B^*B)^{1/2}
\eeq admit unique solutions $S,\,C\in B(\H)$ also verifying that 
$$
\ker (A^*A+B^*B)^{1/2} = R(\, (A^*A+B^*B)^{1/2})\orto
 \stackrel{\eqref{los keres}}\inc R(S^*)\orto \cap R(C^*)\orto = 
\ker S\cap \ker C   \ .
$$
Hence, if $z\in R(\,(A^*A+B^*B)^{1/2})$ and $x\in\H$ is such that 
$z=(A^*A+B^*B)^{1/2}x$ then
$$
\barr{rl}
\langle (S^*S+C^*C)z,z\rangle&
=\|S(A^*A+B^*B)^{1/2}x\|^2+ \|C(A^*A+B^*B)^{1/2}x\|^2 \\& \\
& \stackrel{\eqref{eq op eq1}}= \|A\,x\|^2 +\|B\,x\|^2 = \api (A^*A+B^*B)\,x\coma x\cpi =\|z\|^2 
\ .\earr
$$ 
On the other hand, if 
$z\in \ker (AA^*+BB^*)^{1/2}$ then $(S^*S+C^*C)z=0$. Hence, if we let $P$ denote the orthogonal projection onto the closure of $R((A^*A+B^*B)^{1/2})$ then the previous facts show that
$\langle (S^*S+C^*C)z,z\rangle=\langle P z,z\rangle$, for $z\in\H$; thus, $S^*S+C^*C=P$.
Moreover
$$
AEB^*=S \,(A^*A+B^*B)^{1/2}\,E\,(A^*A+B^*B)^{1/2}\,C^*\,.
$$ 
We now apply Corollary \ref{pro para gen AM comp} 
and we get Eq. \eqref{eq con spread2}.
\end{proof}

\begin{rem}
Let $A$, $B$ and $E$ be as in Theorem \ref{teo generalized AG ineq}, but 
assume that $E\ge0$. Since $s\, \big(\,(A^*A+B^*B)^{1/2}\,E\,(A^*A+B^*B)^{1/2}\,\big) = 
s\,\big(\,E\rai\,(A^*A+B^*B)\,E^{1/2}\,\big)$ and $\sprm(C)=s(C)$ for $C\in K(\H)^+$,  
we can refomulate Eq. \eqref{eq con spread2} in the positive compact case as 
\beq\label{eq con spread3} 
2\, s(A\,E\,B^*)\prec_w s\,\big(\,E\rai\,(A^*A+B^*B)\,E^{1/2}\,\big) \ .
\eeq 
Then  Eq. \eqref{Ec Bhakitta 2} and its consequence Eq. \eqref{eq aud20}, being  entry-wise inequalities, are stronger than Eq. \eqref{eq con spread3} in the positive case. As in the previous 
inequalities of this paper,  
Eq. \eqref{eq con spread2} is a substitute of 
Eq. \eqref{eq aud20} in the self-adjoint non positive case. 

\pausa
Nevertheless, as it happens with Eq. \eqref{Ec BhaDav}, for the general self-adjoint case 
the submajorization relation  \eqref{eq con spread2} cannot be improved to an entry-wise inequality as \eqref {eq aud20}. 
As in previous examples, we shall use that a matrix $A\in \mathcal{M}_n(\C)$ can be embedded as a finite rank operator, which allow us to 
build counterexamples using matrices. 
Consider $$A=\bm{ccc}
	1&0&-1\\
	0&1&0\\
	1&0&1
\em
\,, \ \ \ B=\bm{ccc}
-1&0&1/2\\
0&1&0\\
1/2&0&1
\em 
 \, \py E=\bm{ccc}
1&0&2\\0&1&0\\
2&0&1
\em
\,.$$
In this case  $F=A^*A+B^*B=\bm{ccc}
 \frac{13}4&0&0\\
0&2&0\\
0&0&\frac{13}4
\em
$ is such that 
$$
\barr{rl}
\la(F^{\frac12}E\,F^{\frac12}) &
= \big(\,(\dots \coma 0 \coma \frac{-13}4) \coma (\frac{39}4 \coma 2 \coma 0 \coma \dots) \, \big) 
\ \ \text{and} \  \ \sprm(F^{\frac12}E\,F^{\frac12})=(13 \coma 2 \coma 0 \coma \dots) .
\earr$$
Moreover $s(A\,E\,B^*)\approx(4,74 \coma 1,58 \coma 1 \coma 0....)$, which shows that 
\beq
3,16 \approx 2\,s_2(A\,E\,B^*)>\sprm_2(F^{\frac12}E\,F^{\frac12})=2\ . \EOEP
\eeq
\end{rem}

\begin{rem}\label{pepepe}
Using  Proposition \ref{pro para gen AM}, it is not difficult to prove that 
in the general case 
we can get a weaker version 
of Eq. \eqref{eq con spread2}: 
Let $A,\,B\in B(\H)$ and let $E\in B(\H)^{sa}$. Then 
\beq\label{eq con spread5} 
s(A\,E\,B^*)\prec_w \frac 1 2 \ \sprm \, \big(\,(A^*A+B^*B)^{1/2}\,E\,(A^*A+B^*B)^{1/2}
\oplus 0_\H\,\big)\,.
\eeq 
\EOE
\end{rem}

\subsection{The matrix case}
Notice that, by Eq. \eqref{dim fin} and Eq. \eqref{defi spread compactos}, 
the definition of spectral spread of self-adjoint operator given in \ref{defi spreads}
essentially coincides with the matrix spread defined in \cite{AKFEM}
and considered in \cite {MSZ2}; the main difference is that the spread of self-adjoint matrices is a finite vector.
Nevertheless, if we embed a self-adjoint matrix $A$ as the finite rank operator $A\oplus 0$ by adding a zero block (as we have done in the examples) we notice that the spread of $A$ differs from the spectral spread of the compact operator $A\oplus 0$. 
For these reasons we consider all Hilbert spaces in this paper to have infinite dimension, in order to maintain  consistency.

\pausa
On the other hand, several statements of our present work hold for the finite dimensional case, 
by making slight adaptations of the proofs given here. Indeed, 
Theorem \ref{teo doblesprm oplus} is still valid for matrices and it 
is stronger than every statement about commutators given in \cite {MSZ2}. 
Also Proposition \ref{pro para gen AM}, Corollary \ref{coro AG} and Eq. \eqref{eq con spread5} in Remark \ref{pepepe}
hold in the matrix case, where no results about AGM-type inequalities using 
spectral spread were known to the best of our knowledge. 

\pausa Nevertheless, there are some results that do not hold for the finite dimensional case, which are 
 those statements restricted to compact operators that use the equality 
$\sprm(A\oplus 0_\H) = \sprm(A)$; notice that this last identity does not hold for self-adjoint matrices $A$ (see for example 
Corollary \ref{pro para gen AM comp} and Theorem \ref{teo generalized AG ineq}). 
Observe that  the counterexamples given in Remark \ref{rem mas cero} 
also work if $\dim \H <\infty$.

\subsection{On the equivalence of inequalities for the spectral spread}\label{sec 53}

\pausa
In this last section we show that several of the main results in this work are equivalent. It is worth pointing out that each reformulation has a quite different appeal. Indeed, notice that the statements involve the key result on the spectral spread (Theorem \ref{teo valecon2op}),
commutator inequalities and AGM-type inequalities, all with respect to submajorization. 
In the list of equivalent inequalities below we include a new inequality (item 4.) which is a Zhan type inequality  for the singular values of the difference of self-adjoint operators (see \cite{Zhan00} and also \cite{Zhan02,Zhan04}).

\begin{teo}
The  following inequalities are equivalent:
\ben
\item  $2\,s(P\,E\,(I-P))\prec_w \sprm(E)$, for every $E \coma P \in B(\H)^{sa}$, with $P=P^2$.
\item $s(EF-FE)\prec_w \frac 12 \ \sprm(E\oplus E) \, \cdot \, \sprm(F\oplus F)$, for every $E,\,F\in B(\H)^{sa}$.
\item $s(EX-XF)\prec_w \sprm(E\oplus F)\, \cdot \,s(X)$, for every $E,\,F\in B(\H)^{sa}$ and $ X\in B(\H)$. 
\item $s(E-F)\prec_w \sprm(E\oplus F)$, for every $E,\,F\in B(\H)^{sa}$.
\item 
Given $C,\,S\in B(\H)$  such that $C^*C+S^*S = I$, 
and $E\in B(\H)^{sa}$,  then
$$
2\,s(SEC^*)\prec_w \sprm(E \oplus 0_\H)\,.
$$

\een
\end{teo}
\begin{proof}
By inspection of the proof of Proposition \ref{pro para gen AM} 
we see that $1.\rightarrow 5$. 
On the other hand, as in  Remark \ref{caso la}, 
given $E\in B(\H)^{sa}$, there exists $\la\in \R$ such that $\la_{-i}(E)\le \la \le \la_i(E) $ 
for every $i \in \N$. Denote  $E_\la = E-\la\, I$. Then
 $\sprm(E) = \sprm(E_\la) \stackrel{\eqref{sumando B}}= \sprm(E_\la \oplus 0_\H )$.
Hence
$$ 
2\,s(PE(I-P)\,) =2\,s(PE_\la (I-P)\,)   \stackrel{5.}{\prec_w} \sprm(E_\la \oplus 0_\H ) =\sprm(E_\la) =  \sprm(E)\ ,
$$ that shows that $5\rightarrow 1$. 
By inspection of the proofs of Theorem \ref{teo doblesprm oplus} and its 
Corollary \ref{cor commutator viejo} we see that 
$1.\rightarrow 2. \rightarrow 3$. 

\pausa On the other hand, if we let $X=I$ in item 3. we get item 4. Thus, $3\rightarrow 4$.

\pausa To prove  $4. \rightarrow 1.$ we can assume that $\cH = \cK\oplus \cK$ and that
$P = P_{\cK\oplus \trivial}\,$. Then 
$$
E=\bm{cc}E_1& B\\B^*&E_2\em \barr {c}\cK \\ \cK\earr \peso{with} s(PE(I-P))=s(B)\in c_0(\N)\da\,.
$$
Let $B =U |B| $ be the polar decomposition of $B$. Then $U\in B(\cK)$ is a partial isometry. 
If we construct the partial isometry $W\in B(\H) = B(\cK\oplus \cK)$ given by
$$
W= \bm{cc}U&0\\0&I_\cK\em \peso{then}
W^*EW = \bm {cc}U^*E_1U& U^*B \\ B^*U &E_2\em  = \bm {cc}U^*E_1U& |B| \\ |B| &E_2\em \ .
$$
Then, by item 4. in Proposition  \ref{pro rejunte spread2} we get that $\sprm(W^*EW)\prec_w \sprm(E)$ and $s(B)=s(|B|)$.
 Hence,  in order to show item 1. we can assume that $B\in K(\cK)^{sa}$. 
In this case, taking the self-adjoint unitary operator 
$R=\bm{cc}0&I\\I&0\em\in  B(\cK\oplus \cK)$, we have that
$$
RER 
= \bm {cc}E_2& B\\ B &E_1\em 
\implies
\frac{E+RER}{2}=\bm {cc}\frac{E_1+E_2}2 &B\\ B &\frac{E_1+E_2}2\em  \ . 
$$
By item 3. in Theorem \ref{teo colecta} (Weyl inequality) and item 4. in Proposition \ref{pro rejunte spread2}  
$$\la\left(\frac{E+RER}{2}\right) \prec \frac{\la(E)+\la(RER)}2=\la(E)
{\implies} 
\sprm \left(\frac{E+RER}{2}\right)\prec_w\sprm(E)\,.$$ Therefore, 
in order to show item 1. we can assume that $B\in K(\cK)^{sa}$ and $E_1=E_2\,$.

\pausa
Take now $Z = \frac 1{\sqrt2}\bm{cc}I&I\\-I&I\em\in B(\cK\oplus \cK)$ 
that is a unitary operator. Since now 
\beq\label{ZAZ}
E= \bm {cc}E_1&B\\ B &E_1\em \implies
Z^*EZ = 
\bm{cc}E_1-B&0\\0&E_1+B\em 
\ \, , \ \,  \sprm(Z^*EZ)=\sprm(E)\,.
\eeq
Hence, using item 4 we get that 
$$2\,s(B) = s\big(\, [E_1-B]-[E_1+B]\,\big) \,
\prec_w \, \sprm  (\, [E_1-B]\oplus [E_1+B]\,\big)  
\stackrel{\eqref{ZAZ}}{=} \sprm(E)\,.$$
This shows that $4\rightarrow 1$. and we are done.
\end{proof}

\pausa
In  the compact case we can add another equivalent inequality: 
\begin{pro}
The  following inequalities are equivalent: 
\ben
\item   For every $E \in K(\H)^{sa}$ and $ P \in B(\H)^{sa}$, with $P=P^2$,  
$$2\,s(P\,E\,(I-P))\prec_w \sprm(E) \ .$$

\item For every $E \in K(\H)^{sa}$,  and $A  \coma B\in B(\H)$,  
$$s(A\,E\,B^*)\prec_w \frac 1 2 \ \sprm((A^*A+B^*B)^{1/2}\,E\,(A^*A+B^*B)^{1/2}) \ .$$

\een
\end{pro}
\proof
By inspection of the proofs  of Corollary \ref{pro para gen AM comp} and Theorem \ref{teo generalized AG ineq} 
we see that $1.\rightarrow 2.$
On the other hand, if we take $A=P$ and $B=I-P$ in $2,$ we see that $A^*A+B^*B=I$ and we recover item $1$. 
\QED

\medskip

\section{Appendix}

First, we collect several well known results about majorization, used throughout our work.
For detailed proofs of these results and general references in submajorization theory see \cite{GoKre}. 
In what follows we let $\M$ denote $\N$ or $\Z_0 \igdef \Z \setminus \trivial$.

\begin{lem}\label{lema desigrespetanmayo}\rm 
Let ${\bf x}\coma {\bf y}\coma  {\bf z}\coma {\bf  w} \in \ell^\infty(\M)\cap \R^\M$ be real sequences. Then,
\ben
\item 
${\bf x}+{\bf y}\prec {\bf x}^{\daua}+{\bf y}^{\daua}$;
\item If ${\bf x}\prec {\bf y}$ then $|{\bf x}|\prec_w |{\bf y}|$;
\item[3.] If ${\bf x}\prec {\bf z}$, ${\bf y}\prec {\bf w}$ with ${\bf z}={\bf z}\da$ and ${\bf w}={\bf w}\da$ then, $ {\bf x}+ {\bf y}\prec {\bf z}+{\bf w}$. 
\een
If we assume that $\M=\N$, then 
\ben 
\item[4.] If  ${\bf x}\prec_w {\bf y}$ and ${\bf z}\prec_w {\bf w}$ then $({\bf x},{\bf z})\prec_w({\bf y},{\bf w})$.
\een
If we assume further that ${\bf x}\coma {\bf y}\coma  {\bf z}\in \ell^\infty(\N)$ are non-negative sequences then,
\ben
\item[5.] $ {\bf x}\, \cdot \, {\bf y}\prec_w {\bf x}\da\, \cdot \, {\bf y}\da$;
\item[6.] If ${\bf x}\prec_w {\bf y}$ and ${\bf y}={\bf y}\da,\, {\bf z}={\bf z}\da$  then 
${\bf x}\, \cdot \, {\bf z}\prec_w {\bf y}\, \cdot \, {\bf z}$.\qed
\een
\end{lem}

\pausa
The following results about finite vectors appear in \cite[Chapter II]{bhatia}.  
\begin{lem} \label{dos de Bh}
Let $\x \coma \y \coma \z\in \R^n$ 
\ben
\item  If $\x \coma \y \in \R_{\ge 0}^n\,$ , then $\x\da \, \cdot \,  \y \ua \prec \x \, \cdot \,  \y \prec \x\da \, \cdot \, \y \da $.
\item If $\x \coma \y \in (\R^n)\da$  and $\z\in (\R_{\ge 0}^n)\da\,$, then 
$\x\prec_w\y \implies \suml _{i\in \I_n} x_i \, z_i \le \suml _{i\in \I_n} y_i \, z_i$\,.
\QED
\een
\end{lem}

\pausa
Next we consider the following useful fact.

\begin{pro}\label{hat trick como en el futbol}\rm
Let $E\in B(\H)^{sa}$ and set
$\hat{E}=\bm{cc} 0&E\\  E^*&0
\em\in B(\H\oplus \H)^{sa}$. Then, 
\beq
\la(\hat E)= \big(\, s(E)\coma -s(E^*)\,\big)^{\daua}
=\big(\,s(E)\coma -s(E)\, \big)^{\daua}
\,.
 \QEDP
\eeq
\end{pro}

\pausa
Finally, we state a well known result of R. Douglas \cite{Dou66} which contains criteria for the factorization of operators that we will need in the sequel.
\begin{teo}
\label{dou} \rm
Let $A \coma B \in B(\H)$. Then the following conditions are equivalent: 
\ben
  \item $R(A)\subseteq R(B)$.
  \item There exists $\la \in \R_+$	 such that  
	$AA^*\leq \lambda \,  BB^*$. 
  \item There exists $C\in B(\H)$ such that $A=BC$.  
  
\een
In this case, there exists an unique 
\beq\label{los keres}
C\in  B(\H ) \peso{such that} A=BC \peso{and} R(C)
\subseteq \overline{R(B^*)}= \ker B\orto \ .
\eeq
\QED
\end{teo}


{\scriptsize
\begin{thebibliography}{99}

\bibitem{Alda12} H. Albadawi, Singular value and arithmetic-geometric mean inequalities for operators. Ann. Funct. Anal. 3 (2012), no. 1, 10-18.
\bibitem{ACS97} E. Andruchow, G.  Corach, D. Stojanoff, Geometric operator inequalities. Linear Algebra Appl. 258 (1997), 295-310.


\bibitem{Audeh20a} W. Audeh, Generalizations for singular value and arithmetic-geometric mean inequalities of operators. J. Math. Anal. Appl. 489 (2020), no. 2, 124184, 8pp.

\bibitem{AuKitt} W. Audeh, F. Kittaneh, Singular value inequalities for compact operators. Linear Algebra Appl. 437 (2012), no. 10, 2516-2522. 


\bibitem{AMRS07} J. Antezana, P. Massey, M. Ruiz, D. Stojanoff, The Schur-Horn theorem for operators and frames with prescribed norms and frame operator. Illinois J. Math. 51 (2007), no. 2, 537-560. 

\bibitem{bhatia} Bhatia, R., Matrix analysis, 169, Springer-Verlag, New York, 1997. 

\bibitem{BhaDav} R. Bhatia and C. Davis, More matrix form of the arithmetic geometric mean inequality.
 SIAM J. Matrix Anal. Appl. 14 (1993), 132-136.


\bibitem {bha-Kitt90} R. Bhatia, F. Kittaneh, On the singular values of a product of operators. SIAM J. Matrix Anal. Appl. 11 (1990), no. 2, 272-277.

\bibitem {bha-Kitt08} R. Bhatia, F. Kittaneh, The matrix arithmetic-geometric mean inequality revisited. Linear Algebra Appl. 428 (2008), no. 8-9, 2177-2191.

\bibitem {CPR} G. Corach, H. Porta, L. Recht, An operator inequality. Linear Algebra Appl. 142 (1990) 153–158.

\bibitem {CPR2} G. Corach, H. Porta, L. Recht,  Geodesics and operator means in the space of positive operators. Internat. J. Math. 4 (1993), no. 2, 193-202.

\bibitem {CPR3} G. Corach, H. Porta, L. Recht, The geometry of the space of self-adjoint invertible elements in a $C^*$-algebra. Integral Equations Operator Theory 16 (1993), no. 3, 333-359.
1

\bibitem{Dou66} R.G. Douglas, On majorization, factorization, and range inclusion of operators on Hilbert space. Proc. Amer. Math. Soc. 17 (1966), 413-415.

\bibitem{TFa} T. Fack, Sur la notion de valeur caractéristique, J. Operator Theory 7:2 (1982), 307-333.

\bibitem{TFaKos} T. Fack and H. Kosaki, Generalized s-numbers of $\tau$-measurable operators, Pacific J. Math. 123:2 (1986), 269-300.

\bibitem{GoKre} I.C. Gohberg, M.G. Krein, Introduction to the theory of linear nonself-adjoint operators in Hilbert space. Am. Math. Soc., Providence (1969). 

\bibitem{Hir05} O. Hirzallah, Inequalities for sums and products of operators.  Linear Algebra Appl. 407 (2005) 32-42.

\bibitem{Hirz09} O. Hirzallah, Commutator inequalities for Hilbert space operators. Linear Algebra Appl. 431 (2009), no. 9, 1571-1578. 


\bibitem{Hir16} O. Hirzallah, Singular values of convex functions of operators and the arithmetic-geometric mean inequality. J. Math. Anal. Appl. 433 (2016), no. 2, 935-947.

\bibitem{KitHir} O. Hirzallah, F. Kittaneh, Singular values,norms, and commutators. Linear Algebra and its Applications 432 (2010) 1322–1336.

\bibitem{Kad04} R. V. Kadison, Non-commutative conditional expectations and their applications, pp. 143-179 in Operator algebras, quantization, and noncommutative geometry (Baltimore, MD, 2003), edited by R. S. Doran and R. V. Kadison, Contemp. Math. 365, Amer. Math. Soc., Providence, RI, 2004.

\bibitem{Kitt}
 F. Kittaneh, Inequalities for commutators of positive operators. J. Funct. Anal. 250 (2007), no. 1, 132-143.

\bibitem{Kitt2008} F. Kittaneh, Norm inequalities for commutators of Hermitian operators. Integral Equations Operator Theory 62 (2008), no. 1, 129-135.

\bibitem{Kitt09} F. Kittaneh, Singular value inequalities for commutators of Hilbert space operators. Linear Algebra Appl. 430 (2009), no. 8-9, 2362-2367. 

\bibitem{AKFEM} A.V. Knyazev, M.E. Argentati, 
Rayleigh-Ritz majorization error bounds with applications to FEM. SIAM J. Matrix Anal. Appl. 31 (2009), no. 3, 1521-1537. 

\bibitem{MK} A.S.Markus, The eigen- and singular values of the sum and product of linear operators, 1964 Russ. Math. Surv. 19 R02


\bibitem{MSZ1} P. Massey, D. Stojanoff, S. Zárate, Majorization bounds for Ritz values of self-adjoint matrices. SIAM J. Matrix Anal. Appl. 41 (2020), no. 2, 554-572. 

\bibitem{MSZ2} P. Massey, D. Stojanoff, S. Zárate, The spectral spread of Hermitian matrices. Linear Algebra Appl. 616 (2021), 19-44. 

\bibitem{MSZ3} P. Massey, D. Stojanoff, S. Zárate, Absolute variation of Ritz values, principal angles and spectral spread. SIAM J. Matrix Anal. Appl., accepted for publication (in press, 2021).

\bibitem{Dpetz} D. Petz, Spectral scale of self-adjoint operators and trace inequalities, J. Math. Anal. Appl. 109:1 (1985), 74-82.

\bibitem{Tao06} Y. Tao, More results on singular value inequalities of matrices. Linear Algebra Appl. 416 (2006), no. 2-3, 724-729.

\bibitem{Zhan00} X. Zhan, Singular values of differences of positive semidefinite matrices. SIAM J. Matrix Anal. Appl. 22 (2000), no. 3, 819-823. 

\bibitem{Zhan02} X. Zhan, Matrix inequalities. Lecture Notes in Mathematics, 1790. Springer-Verlag, Berlin, 2002. 

\bibitem{Zhan04} X. Zhan, On some matrix inequalities. Linear Algebra and its Applications 376 (2004) 299-303. 
\end{thebibliography}}
\end{document}